\documentclass{amsart}

\usepackage{graphicx}
\usepackage[numbers]{natbib}
\usepackage{color,soul}
\usepackage[all]{xy}
\usepackage{subfigure}
\usepackage{geometry,mathtools}
\usepackage{amsthm}
\usepackage{float}
\usepackage{tikz-cd}
\usepackage{extarrows}
\usepackage{extarrows}
\usepackage{multirow}
\usepackage{bm}
\usepackage{amssymb}
\usepackage{soul}
\usepackage{graphicx}
\usepackage{color}
\usepackage{verbatim}
\usepackage{mathtools}

\newtheorem{theorem}{Theorem}[section]

\newtheorem{proposition}[theorem]{Proposition}
\newtheorem{prop}[theorem]{Proposition}
\newtheorem{corollary}[theorem]{Corollary}
\newtheorem{lemma}[theorem]{Lemma}

\theoremstyle{definition}
\newtheorem{definition}[theorem]{Definition}

\theoremstyle{remark}
\newtheorem{remark}[theorem]{Remark}

\numberwithin{equation}{section}

\begin{document}

\title{Rational genus and Heegaard Floer homology}

\author{Zhongtao Wu}
\address{Department of Mathematics, The Chinese University of Hong Kong}
\email{ztwu@math.cuhk.edu.hk}
\author{Jingling Yang}
\address{School of mathematics, Xiamen University}
\email{yangjingling@xmu.edu.cn}
\thanks{} 




\begin{abstract}
Turaev defined a function on the first homology of a rational homology 3-sphere $Y$ as the minimal rational Seifert genus of all knots in this homology class. Ni and the first author discovered a lower bound of this function using the Heegaard Floer $d$-invariant and showed that Floer simple knots are rational Seifert genus minimizers. In this paper, we give a simple reproof of the above results. We then define a version of rational slice genus for knots in the product 4-manifold $Y\times I$ and investigate the analogous minimal genus problem. We prove the same lower bound in terms of the $d$-invariant formula and the same genus minimizers given by Floer simple knots.
\end{abstract}

\maketitle

\section{Introduction}

In this paper, we reprove Ni-Wu's result about a 3-dimensional minimal genus problem for rationally null-homologous knots, and then study an analogous 4-dimensional minimal genus problem for rationally null-homologous knots. 

Let $K$ be a knot of order $p$ in a rational homology 3-sphere $Y$. Let $M=Y-N^{\circ}(K)$, where $N^{\circ}(K)$ denotes the interior of $N(K)$, a solid torus neighborhood of $K$. Note that the inclusion map
\[i_{\ast}:H_{1}(\partial M; \mathbb{Z}) \rightarrow H_1(M; \mathbb{Z})\]
has kernel isomorphic to $\mathbb{Z}$, which is generated by $k \lambda_r$ for some primitive class $\lambda_r \in H_1(\partial M; \mathbb{Z})$ and positive integer $k$. The class $\lambda_r$ is well-defined up to orientation, which therefore gives a well-defined slope, called the \textit{rational longitude} for $K$. A \textit{rational Seifert surface} for $K$ is a properly embedded, compact, connected, oriented surface $S \subset M$ with $[\partial S]=k\lambda_r \in H_1(\partial M; \mathbb{Z})$. The \textit{rational Seifert genus} for $K$ is defined as:

\[\|K\|_{Y}:=\mathop{\min}\limits_{S} \dfrac{-\chi(S)}{2|[\mu] \cdot [\partial S]|}= \mathop{\min}\limits_{S} \dfrac{-\chi(S)}{2p},\]
where $\mu$ denotes the meridian of $K$ and $S$ is any rational Seifert surface for $K$, c.f. \cite{CG}\cite{LRS}. We will show later (in Section \ref{4-dimensional rational slice genus}) that $|[\mu] \cdot [\partial S]|=p$, which then implies that $p$ is a multiple of $k$.


Minimizing the rational Seifert genus gives the so-called Tureav function \cite{Tur2} defined on the torsion subgroup of $H_1(Y; \mathbb{Z})$, that is,  
\[\Theta(a):=\mathop{\min}\limits_{K \subset Y, [K]=a} 2 \|K\|_{Y}\] for each torsion class $a \in H_1(Y; \mathbb{Z})$.  In \cite{NW}, Ni and the first author gave a novel bound of $\Theta(a)$ in terms of the Heegaard Floer $d$-invariant, and it was also shown that Floer simple knots are rational Seifert genus minimizers in the respective homology classes.  As the original proof was slightly technical that involved ad hoc combinatorial argument, the first half of this paper is devoted to gaining a deeper understanding of the set 
$\{d(Y,\mathfrak{s})-d(Y,\mathfrak{s}+PD[K])\, | \, \mathfrak{s}\in {\rm Spin^c}(Y)\}$ by exhibiting it as the Alexander grading of certain naturally arising middle-level relative $\rm Spin^c$ structures, and using this approach to give a significantly simplified proof of the following statement.  

\begin{theorem}{\rm \cite[Theorem 1.1, Theorem 1.2]{NW}}
\label{Rasmussen conjecture NW}
Let $K$ be a knot in a rational homology 3-sphere $Y$. Then 
\begin{equation}
\label{NiWuoldresult}
2\|K\|_Y + 1 \geq \mathop{\max}\limits_{\mathfrak{s} \in {\rm Spin^c}(Y)} \left\{ d(Y,\mathfrak{s})-d(Y,\mathfrak{s}+PD[K]) \right\}.
\end{equation}
Floer simple knots attain the equality, and consequently they are rational Seifert genus minimizers in their homology classes.
\end{theorem}

\begin{remark}
As the right hand side of (\ref{NiWuoldresult}) only depends on the manifold $Y$ and the homology class of $K$, it gives a lower bound for $1+\Theta(a)$ for the homology class $a=[K]$.
\end{remark}

In particular, Theorem \ref{Rasmussen conjecture NW} implies the {\it Rasmussen conjecture}, which claims that simple knots in lens spaces are genus minimizers in their homology classes \cite{RasNotation}. Our new proof also offers new insight to the uniqueness issue of genus minimizers, which appears to be a rather subtle question.  While Baker \cite{Baker} proved the uniqueness of a genus minimizer when the genus is less than $\frac{1}{4}$, Greene and Ni \cite{GN} disproved the uniqueness in general.  For the purpose of proving the famous Berge conjecture on lens space surgeries \cite{Berge}, one needs to determine whether a simple knot with genus less than $\frac{1}{2}$ is always the unique genus minimizer in its homology class \cite{Hedden}\cite{RasNotation}.  This remains an open question to this day.  


\bigskip
Similarly, we can define a 4-dimensional genus. Given $K$, a \textit{Seifert framed rational slice surface} is defined to be a compact, connected, oriented surface $F$ embedded in $Y \times [0,1] - N(K)$ such that $\partial F=F \cap \partial N(K)$ and  $[\partial F]=[\partial S]=k \lambda_r$. Here, $N(K)$ denotes a solid torus neighborhood of $K$ in $Y \times \{1\}$ and $S$ is any rational Seifert surface for $K$. The corresponding \textit{rational slice genus}, or more precisely, \textit{rational slice genus relative to the rational longitude} of $K$ is then defined as:
\[\|K\|_{Y \times I}^{\partial }:=\mathop{\min}\limits_{F} \dfrac{-\chi(F)}{2|[\mu] \cdot [\partial F]|}=\mathop{\min}\limits_{F} \dfrac{-\chi(F)}{2p},\]
where $F$ is any Seifert framed rational slice surface for $K$. 


Clearly, $\|K\|_{Y \times I}^{\partial }\leq \|K\|_{Y}$. It is natural to ask whether the same $d$-invariant bound (\ref{NiWuoldresult}) remains true when we replace the rational Seifert genus by the rational slice genus. Results in Levine-Ruberman-Strle \cite{LRS} imply the case for $K$ of order 2. We are able to prove the general case:   


\begin{theorem}
\label{d-dleqprsg}
 Let $K$ be a knot in a rational homology 3-sphere $Y$. Then
 \begin{equation*}
 2\|K\|_{Y \times I}^{\partial } + 1 \geq \mathop{\max}\limits_{\mathfrak{s} \in {\rm Spin^c}(Y)} \left\{ d(Y,\mathfrak{s})-d(Y,\mathfrak{s}+PD[K]) \right\}.
 \end{equation*}
Floer simple knots attain the equality, and consequently they are rational slice genus minimizers in their homology classes.
\end{theorem}


On our way to prove Theorem \ref{d-dleqprsg}, we also generalize the notion of $\nu^+$ invariant to each $\rm Spin^c$ structure $\mathfrak{s}$ in a rational homology 3-sphere. See Section \ref{Spinc structures and Alexander grading} for its precise definition. Similar to the classical slice genus bound $\nu^+(K)\leq g_4(K)$ for knots in the 3-sphere, we prove:
\begin{theorem}
\label{nuleqprsg}
For any knot $K$ in a rational homology 3-sphere $Y$ and $\mathfrak{s}\in {\rm Spin^c}(Y)$, 
$$\nu^+_\mathfrak{s}(Y, K) \leq \|K\|_{Y \times I}^{\partial} + \frac{1}{2}.$$
Consequently,
$$\nu^+(Y, K)=\mathop{\max}\limits_{\mathfrak{s} \in {\rm Spin^c}(Y)} \left\{ \nu^+_{\mathfrak{s}} (Y,K)\right\} \leq \|K\|_{Y \times I}^{\partial} + \frac{1}{2}.$$
In particular, Floer simple knots attain the equality.
\end{theorem}


In Corollary \ref{nud} below, we prove that for each $\rm Spin^c$ structure $\mathfrak{s}$,
\begin{equation}
\label{nugeqd-d}
\nu^+_\mathfrak{s}(Y, K) \geq  \frac{1}{2}d(Y,\mathfrak{s})-\frac{1}{2}d(Y,\mathfrak{s}+PD[K]),
\end{equation}
and this helps establish the $d$-invariant slice genus bound in Theorem \ref{d-dleqprsg}.

\bigskip
To conclude the introduction, we outline our approach to the 4-dimensional case below.
\begin{enumerate}



\item  First, we study those knots whose rational longitude is a framing and apply an adjunction inequality type argument similar to \cite{HR2}.  More precisely, we study the cobordism of $Y \times I$ with a 2-handle attaching along the rational longitude $\lambda_r$ together with a decomposition using the embedded surfaces obtained from the capped-off Seifert framed rational slice surfaces of $K$.  
If we denote $\lambda_r$-surgery along $K$ as $Y_{\lambda_r}$, the adjunction inequality then gives a sufficient condition for the vanishing of the cobordism map from $HF(Y)$ to $HF(Y_{\lambda_r})$.

\item Next, we study the cobordism map from $Y$ to $Y_{\lambda_r}$ using the mapping cone formula.  In particular, we introduce Rasmussen's notation of $+$, $-$, $\circ$, $\ast$, and seek certain ``patterns'' that necessarily lead to (non)vanishing maps. At the end, we are able to understand the vanishing properties of the cobordism map from $Y$ to $Y_{\lambda_r}$ in terms of the $V$ and $H$ invariants, which is by definition related to the $\nu^+$ invariant in Theorem \ref{nuleqprsg}.


\item Finally, we apply Ozsv\'{a}th-Szab\'{o}'s trick of Morse surgery to understand arbitrary knots $K \subset Y$.  

\end{enumerate}

\subsection{Organization.} The paper is organized as follows. In Section \ref{Spinc structures and Alexander grading}, we make our key observation about Ni-Wu's $d$-invariant genus bound by interpreting it as the Alexander grading of certain distinguished relative $\rm Spin^c$ structure. This leads to a concise 1-page reproof in Section \ref{Rasmussen conjecture} of the Rasmussen conjecture on the 3-dimensional minimal genus problem. In Section \ref{4-dimensional rational slice genus}, we apply the adjunction inequality to analyze the vanishing condition for the cobordism map from $Y$ to $Y_{\lambda_r}$. In Section \ref{Mapping cone formula and Rasmussen notation}, we first recall the mapping cone formula for rationally null-homologous knots and rational longitude surgeries and then introduce Rasmussen's notation to understand the vanishing properties of the cobordism map. In Section \ref{The rational slice genus bound for knots whose rational longitude is a framing}, we show our rational slice genus bound for knots whose rational longitude is a framing. Section \ref{The proof of main theorems} proves the bound for all knots. 

\subsection{Notation.} We fix some notations that will be used throughout the paper. Unless noted otherwise, we let $K$ be a knot of order $p$ in a rational homology 3-sphere $Y$. We use $N(K)$ to denote a solid torus neighborhood of $K$ in $Y$, and $N^{\circ}(K)$ denotes the interior of $N(K)$. Let $\mu$ be the meridian of $K$ and $\lambda$ a framing which is an embedded curve on $\partial N(K)$ intersecting $\mu$ transversely once. Here, $\lambda$ naturally inherits an orientation from $K$. Let $\lambda_r$ denote the rational longitude of $K$, which may or may not be a framing.  Let $S$ denote a rational Seifert surface for $K$, and $F$ denote a Seifert framed rational slice surface for $K$. The capped-off $S$ and $F$ (if existed) are denoted by $\widehat{S}$ and $\widehat{F}$ respectively. Given any surgery slope $\alpha$, the 3-manifold obtained from $\alpha$-surgery along $K$ is denoted by $Y_{\alpha}(K)$, which is sometimes abbreviated as $Y_{\alpha}$ when the surgered knot $K$ is clear.

\subsection{Acknowledgements}
We would like to thank Yi Ni, Jennifer Hom, Matthew Hedden, and Katherine Raoux for helpful discussions. The first author is partially supported by a grant from the Research Grants Council of Hong Kong Special Administrative Region, China (Project No. 14301819). The second author is supported by Fundamental Research Funds for the Central Universities (Project No. 20720230026).

\section{$\rm Spin^c$ structures and Alexander grading}
\label{Spinc structures and Alexander grading}

In this section, we recall the background in Heegaard Floer theory and prove some key lemmas and propositions needed in the following sections to reprove Ni-Wu's old result. 

\subsection{Heegaard Floer homology of large surgeries and knot Floer complexes.}
Let $K$ be a knot of order $p$ in a rational homology 3-sphere $Y$.
Let $\mu$ be the meridian of $K$ and $\lambda$ a framing which is an embedded curve on $\partial N(K)$ intersecting $\mu$ transversely once. Here, $\lambda$ naturally inherits an orientation from $K$. 

We denote the set of relative $\rm Spin^c$ structures over $M=Y-N^{\circ}(K)$ by $\underline{\rm Spin^c}(Y,K)$, which is affinely isomorphic to $H^2(Y,K)$. Recall that for any relative $\rm Spin^c$ structure $\xi \in \underline{\rm Spin^c}(Y,K)$, the \textit{Alexander grading}\footnote{In some literature, the Alexander grading formula has $-[\mu] \cdot [S]$ instead of $+[\mu] \cdot [S]$, e.g.,  \cite{NV}. This is due to opposite knot orientation conventions from doubly-pointed Heegaard diagrams. Our convention agrees with \cite{HL}\cite{OSk}  but differs from \cite{NV}\cite{NW}\cite{OSl}\cite{OSr}.} 
of $\xi$ is defined as:
\[A_{Y,K}(\xi)=\dfrac{\langle c_1(\xi), [S] \rangle + [\mu] \cdot [S]}{2 [\mu] \cdot [S]}=\dfrac{\langle c_1(\xi), [S] \rangle + p}{2p}.  \]

Let $w$ be a vector field on $S^1 \times D^2$ as described in \cite[Subsection 2.2]{OSr} (also see \cite[Subsection 2.2]{NW}), which is also called the distinguished Euler structure in Turaev's book \cite[Chapter \uppercase\expandafter{\romannumeral6}]{Tur}. Gluing this vector field $w$ to a relative $\rm Spin^c$ structure $\xi$ on $M$ gives a natural map:
\[G_{Y, \pm K}: \underline{\rm Spin^c}(Y,K) \rightarrow {\rm Spin^c}(Y) \cong H^2(Y)\]
satisfying
\[G_{Y,\pm K}(\xi+\kappa)=G_{Y, \pm K}(\xi)+i^{\ast}(\kappa),\]
where $\kappa \in H^2(Y,K)$ and $i^{\ast}: H^2(Y,K) \rightarrow H^2(Y)$ is induced from inclusion. Here, $-K$ denotes $K$ with the opposite orientation. We have
\[G_{Y,-K}(\xi)=G_{Y,K}(\xi)+PD[\lambda],\]
where $\lambda$ is the push-off of $K$ inside $M$ using the framing $\lambda$. We call $G_{Y,K}(\xi)$ the \textit{underlying} $\rm Spin^c$ structure of $\xi$ and denote the set of relative $\rm Spin^c$ structures with underlying $\rm Spin^c$ structure $\mathfrak{s}$ by $\underline{\rm Spin^c}(Y,K,\mathfrak{s})$.

For each $\xi \in \underline{\rm Spin^c}(Y,K)$, we can associate to it a $\mathbb{Z} \oplus \mathbb{Z}$-filtered knot Floer complex $C_{\xi}=CFK^{\infty}(Y,K,\xi)$, whose bifiltration is given by $(i,j)=(algebraic, Alexander)$. Let 
\[A^+_{\xi}(K)=C_{\xi}\{\max\{i,j\} \geq 0\} \,\,\, {\rm and} \,\,\, B^+_{\xi}(K)=C_{\xi}\{i \geq 0\},\]
and the complexes 
\begin{align*}
C_{\xi}\{i \geq 0\}& \cong CF^+(Y, G_{Y,K}(\xi)) \cong B^+_{\xi}(K), \\
C_{\xi}\{j \geq 0\}&
\cong CF^+(Y,G_{Y,-K}(\xi)) \cong CF^+(Y,G_{Y,K}(\xi+PD[\lambda])) \cong B^+_{\xi +PD[\lambda]}(K).
\end{align*}
There are two natural projection maps
\[v^+_{\xi}: A^+_{\xi}(K) \rightarrow B^+_{\xi}(K), \quad h^+_{\xi}: A^+_{\xi}(K) \rightarrow B^+_{\xi+PD[\lambda]}(K),\]
which can be identified with certain cobordism maps that will be described below. 

Let $W'_n$ be the two-handle cobordism obtained by turning around the two-handle cobordism from $-Y$ to $-Y_{\lambda+n\mu}(K)$. 
Let $\widehat{S}$ denote the capped-off rational Seifert surface $S$ in $W'_n$, then $H_2(W'_n) \cong \mathbb{Z}$ is generated by $[\widehat{S}]$. Let $D \subset W'_n$ denote the core disk of the two-handle attached to $Y$ together with $K \times I$, and let $D^{\ast} \subset W'_n$ denote the cocore disk. We assume they are oriented to intersect positively. Then $[D]$ and $[D^{\ast}]$ generate $H_2(W'_n, Y)$ and $H_2(W'_n,Y_{\lambda+n\mu}(K))$, respectively. Consider the Poincar\'{e} duals $PD[D] \in H^2(W'_n,Y_{\lambda+n\mu}(K))$ and $PD[D^{\ast}] \in H^2(W'_n, Y)$; by a slight abuse of notation, we also use $PD[D]$ and $PD[D^{\ast}]$ to denote the images of these classes in $H^2(W'_n)$.  Recall $K$ represents a class of order $p$ in $H_1(Y)$, and suppose $[\partial S]=p(\lambda+n \mu)-k\mu$ in $H_1(\partial M)$ for some $k \in \mathbb{Z}$, then $[\widehat{S}]$ maps to $-p[D] \in H_2(W'_n, Y)$ and to $k[D^{\ast}] \in H_2(W'_n,Y_{\lambda+n\mu}(K))$. It follows $[\widehat{S}]^2=-pk$. When $n$ is sufficiently large, $W'_n$ is negative definite.


As in \cite[Proposition 2.2]{OSr}, Ozsv\'{a}th and Szab\'{o} construct a bijection
\[E_{Y,\lambda+n\mu, K}: {\rm Spin^c}(W'_n) \rightarrow \underline{\rm Spin^c}(Y,K),\]
which restricts a ${\rm Spin^c}$ structure on $W'_n$ to the knot complement. Note that $E_{Y,\lambda+n\mu,K}$ depends on $n$. 
Hedden and Levine \cite[Remark 4.3]{HL} show that for any $\mathfrak{v} \in {\rm Spin^c}(W'_n)$, $E_{Y,\lambda+n\mu, K}(\mathfrak{v})$ is the relative $\rm Spin^c$ structure that satisfies:
\begin{align}
G_{Y,K}(E_{Y,\lambda+n\mu, K}(\mathfrak{v}))=\mathfrak{v}|_Y, \,\, {\rm and} \,\,\,  A_{Y,K}(E_{Y,\lambda+n\mu, K}(\mathfrak{v}))=\dfrac{\langle c_1(\mathfrak{v}), [\widehat{S}] \rangle + k}{2p}. \label{HLEA}
\end{align}

The following theorem explains how $v^+_{\xi}$ and $h^+_{\xi}$ can be identified with certain cobordism maps. 

\begin{theorem}{\rm \cite[Theorem 4.1]{OSr}}
Let $K$ be a rationally null-homologous knot in a closed, oriented 3-manifold $Y$, equipped with a framing $\lambda$. Then, for all sufficiently large $n$, there is a map
\[\Xi: {\rm Spin^c}(Y_{\lambda+n\mu}(K)) \rightarrow \underline{\rm Spin^c}(Y,K)\]
with the property that for all $\mathfrak{t} \in {\rm Spin^c}(Y_{\lambda+n\mu}(K))$, the chain complex $CF^+(Y_{\lambda+n\mu}(K), \mathfrak{t})$ is represented by the chain complex 
\[A^+_{\Xi (\mathfrak{t})}(K)=C_{\Xi(\mathfrak{t})}\{{\max} \{i,j\}\geq 0\},\]
in the sense that there are isomorphisms
\[\Psi^+_{\mathfrak{t},n}: CF^+(Y_{\lambda+n\mu}(K), \mathfrak{t}) \rightarrow A^+_{\Xi (\mathfrak{t})}(K).\]
Furthermore, fix $\mathfrak{t} \in {\rm Spin^c}(Y_{\lambda+n\mu}(K))$, and let $\Xi(\mathfrak{t})=\xi$. There are ${\rm Spin^c}$ structures $\mathfrak{v}_{\xi}$, $\mathfrak{h}_{\xi} \in {\rm Spin^c}(W'_n)$ with $E_{Y,\lambda+n\mu, K}(\mathfrak{v}_{\xi})=\xi$, and $\mathfrak{h}_{\xi}=\mathfrak{v}_{\xi}+PD[D]$ with the property that the maps $v^+_{\xi}$ and $h^+_{\xi}$ correspond to the maps induced by the cobordism $W'_n$ equipped with $\mathfrak{v}_{\xi}$ and $\mathfrak{h}_{\xi}$ respectively. More precisely, the following squares commute:\\
\begin{minipage}{0.5\textwidth}
\begin{equation}
\begin{tikzcd}
 CF^+(Y_{\lambda+n\mu}(K), \mathfrak{t}) \arrow[d,"f_{W'_n,\mathfrak{v}_{\xi}}"] \arrow[r,"\Psi^+_{\mathfrak{t},n}"] &   A^+_{\xi}(K) \arrow[d,"v^+_{\xi}"] \\
 CF^+(Y,G_{Y,K}(\xi)) \arrow[r, "\cong"]  & B^+_{\xi}(K)
\end{tikzcd}\label{cdeq1}
\end{equation}
\end{minipage}
\begin{minipage}{0.52\textwidth}
\begin{equation}
\begin{tikzcd}
 CF^+(Y_{\lambda+n\mu}(K), \mathfrak{t}) \arrow[d,"f_{W'_n,\mathfrak{h}_{\xi}}"] \arrow[r,"\Psi^+_{\mathfrak{t},n}"] &   A^+_{\xi}(K) \arrow[d,"h^+_{\xi}"] \\
CF^+(Y,G_{Y,-K}(\xi)) \arrow[r, "\cong"]  & B^+_{\xi+PD[\lambda]}(K).
\end{tikzcd}\label{cdeq2}
\end{equation}
\end{minipage}
\end{theorem}

On homology, $v^+_{\xi}$ and $h^+_{\xi}$ induce maps
\[v^+_{\xi,\ast}: H_{\ast}(A^+_{\xi}(K)) \rightarrow H_{\ast}(B^+_{\xi}(K)), \quad h^+_{\xi,\ast}: H_{\ast}(A^+_{\xi}(K)) \rightarrow H_{\ast}(B^+_{\xi+PD[\lambda]}(K)).\]
The Heegaard Floer homology $HF^+$ of any $\rm Spin^c$ rational homology 3-sphere contains a tower $\mathcal{T}^+=\mathbb{F}[U^{-1},U] \slash U \cdot \mathbb{F}[U]$. Define
\[V^+_{\xi}(K)={\rm rank(ker} \,\, v^+_{\xi,\ast}|_{\mathcal{T}^+}), \quad H^+_{\xi}(K)={\rm rank(ker} \,\, h^+_{\xi,\ast}|_{\mathcal{T}^+}) .\]
 Then $v^+_{\xi,\ast}|_{\mathcal{T}^+}$ and $h^+_{\xi,\ast}|_{\mathcal{T}^+}$ are multiplications by $U^{V^+_{\xi}}$ and $U^{H^+_{\xi}}$, respectively.  
The two non-negative integers $V^+_{\xi}(K)$ and $H^+_{\xi}(K)$ satisfy that for each $\xi \in \underline{\rm Spin^c}(Y,K)$,
\begin{align}
&V^+_{\xi}(K) \geq V^+_{\xi+PD[\mu]}(K) \geq V^+_{\xi}(K)-1, \label{Vd} \\
&H^+_{\xi}(K) \leq H^+_{\xi+PD[\mu]}(K) \leq H^+_{\xi}(K)+1. \label{Hi}
\end{align}


\subsection{The $V$ and $H$-invariants and $\nu^+$-invariant.}

In this section, we study the properties of $V$ and $H$-invariants in rational homology 3-spheres and exhibit the existence of certain relative $Spin^c$ structures $\xi^0_\mathfrak{s}$ at which $V^+_{\xi^0_\mathfrak{s}}=H^+_{\xi^0_\mathfrak{s}}$. It is illuminating to note that the Alexander grading of $\xi^0_\mathfrak{s}$ is precisely given by $\frac{1}{2}d(Y, \mathfrak{s})-\frac{1}{2}d(Y, \mathfrak{s}+PD[K])$, as shown in Proposition \ref{middlersc=d}. 


\begin{lemma}
\label{lemmaV-H}
Let $K$ and $K'$ be two knots in a rational homology 3-sphere $Y$ with $[K]=[K'] \in H_1(Y)$. Then for any relative $\rm Spin^c$ structure $\xi \in \underline{\rm Spin^c}(Y,K)\cong \underline{\rm Spin^c}(Y,K')$,
\[V^+_{\xi}(K)-H^+_{\xi}(K)=V^+_{\xi}(K')-H^+_{\xi}(K')\]
and
\[(V^+_{\xi+PD[\mu]}(K)-H^+_{\xi+PD[\mu]}(K))-(V^+_{\xi}(K)-H^+_{\xi}(K))=-1.\]
\end{lemma}

\begin{proof}
By \cite[Theorem 7.1]{OSc} and commutative diagrams (\ref{cdeq1}) and (\ref{cdeq2}), we have
\begin{align}
d(Y, G_{Y,K}(\xi))-d(Y_{\lambda+n\mu}(K),\mathfrak{t})&=\dfrac{c_1^2(\mathfrak{v}_{\xi})-3\sigma(W'_n)-2 \chi(W'_n)}{4} -2V^+_{\xi}(K), \label{dV}\\
d(Y, G_{Y,-K}(\xi))-d(Y_{\lambda+n\mu}(K),\mathfrak{t})&=\dfrac{c_1^2(\mathfrak{h}_{\xi})-3\sigma(W'_n)-2 \chi(W'_n)}{4} -2H^+_{\xi}(
K). \label{dH} 
\end{align}
Equation (\ref{dH}) $-$ (\ref{dV}) implies
\begin{equation}
V^+_{\xi}(K)-H^+_{\xi}(K)=\frac{1}{2} d(Y, G_{Y,-K}(\xi)) - \frac{1}{2} d(Y,G_{Y,K}(\xi)) +\dfrac{c_1^2(\mathfrak{v}_{\xi})-c_1^2(\mathfrak{h}_{\xi})}{8}, \label{VHd1}
\end{equation}
which is completely determined by homological information.  This proves the first part of the lemma.  

\medskip
Similarly, we have 
\begin{equation}
\begin{split}
&V^+_{\xi+PD[\mu]}(K)-H^+_{\xi+PD[\mu]}(K)\\
&=\frac{1}{2} d(Y, G_{Y,-K}(\xi+PD[\mu])) - \frac{1}{2} d(Y,G_{Y,K}(\xi+PD[\mu])) +\dfrac{c_1^2(\mathfrak{v}_{\xi}+PD[D^{\ast}])-c_1^2(\mathfrak{h}_{\xi}+PD[D^{\ast}])}{8}\\
&=\frac{1}{2} d(Y, G_{Y,-K}(\xi)) - \frac{1}{2} d(Y,G_{Y,K}(\xi)) \\
&\,\,\, +\dfrac{c_1^2(\mathfrak{v}_{\xi})+4[D^{\ast}]^2+4 \langle c_1(\mathfrak{v}_{\xi}) \cup PD[D^{\ast}], [W'_n, \partial W'_n] \rangle-c_1^2(\mathfrak{h}_{\xi})-4[D^{\ast}]^2-4\langle c_1(\mathfrak{h}_{\xi}) \cup PD[D^{\ast}], [W'_n, \partial W'_n]\rangle}{8}\\
&=\frac{1}{2} d(Y, G_{Y,-K}(\xi)) - \frac{1}{2} d(Y,G_{Y,K}(\xi)) \\
& \,\,\, + \dfrac{c_1^2(\mathfrak{v}_{\xi})+ 4 \langle c_1(\mathfrak{v}_{\xi}) \cup PD[D^{\ast}], [W'_n, \partial W'_n] \rangle -c_1^2(\mathfrak{h}_{\xi}) -4\langle c_1(\mathfrak{h}_{\xi}) \cup PD[D^{\ast}], [W'_n, \partial W'_n]\rangle}{8}.
\label{VHd2}
\end{split}
\end{equation}
Consider (\ref{VHd2}) $-$ (\ref{VHd1}),
\begin{equation*}
\begin{split}
(V^+_{\xi+PD[\mu]}&(K)-H^+_{\xi+PD[\mu]}(K))-(V^+_{\xi}(K)-H^+_{\xi}(K))\\
&=\dfrac{4 \langle c_1(\mathfrak{v}_{\xi}) \cup PD[D^{\ast}], [W'_n, \partial W'_n] \rangle -4\langle c_1(\mathfrak{h}_{\xi}) \cup PD[D^{\ast}], [W'_n, \partial W'_n]\rangle}{8}\\
&=\dfrac{4 \langle c_1(\mathfrak{v}_{\xi}) \cup PD[D^{\ast}], [W'_n, \partial W'_n] \rangle -4\langle (c_1(\mathfrak{v}_{\xi})+2PD[D]) \cup PD[D^{\ast}], [W'_n, \partial W'_n]\rangle}{8}\\
&=\dfrac{-8\langle PD[D] \cup PD[D^{\ast}], [W'_n, \partial W'_n] \rangle}{8}\\
&=-1,
\end{split}
\end{equation*}
since $\langle PD[D] \cup PD[D^{\ast}], [W'_n, \partial W'_n] \rangle=1$.
\end{proof}

\begin{prop}
\label{teuxi0s}
Suppose $K$ is a knot in a rational homology 3-sphere $Y$. Then for each $\rm Spin^c$ structure $\mathfrak{s}\in {\rm Spin^c} (Y)$, there exists a unique relative 
$\rm Spin^c$ structure $\xi^0 \in \underline{\rm Spin^c}(Y,K)$ 
with underlying $\rm Spin^c$ structure $\mathfrak{s}$ such that 
$$V^+_{\xi^0}(K)=H^+_{\xi^0}(K).$$
\end{prop}

\begin{proof}
All relative $\rm Spin^c$ structures with the same underlying $\rm Spin^c$ structure differ by multiples of $PD[\mu]$. Thus, Lemma \ref{lemmaV-H} implies the unique existence of a $\xi^0 \in \underline{\rm Spin^c}(Y,K,\mathfrak{s})$ with $V^+_{\xi^0}(K)=H^+_{\xi^0}(K)$. 
\end{proof}

\begin{definition}
Denote $\xi^0_\mathfrak{s} \in \underline{\rm Spin^c}(Y,K, \mathfrak{s})$ the above unique relative $\rm Spin^c$ structure that satisfies  $V^+_{\xi^0_\mathfrak{s}}(K)=H^+_{\xi^0_\mathfrak{s}}(K)$. We will call $\xi^0_\mathfrak{s}$ the \textit{middle relative $Spin^c$ structure} associated to $\mathfrak{s}$.  

\end{definition}

\begin{lemma} \label{vhA}
Suppose $K$ is a knot in a rational homology 3-sphere $Y$. The Alexander grading  
\[ A_{Y,K}(\xi)=\frac{1}{8}c_1(\mathfrak{v}_{\xi})^2-\frac{1}{8}c_1(\mathfrak{h}_{\xi})^2\]
for any $\xi \in \underline{\rm Spin^c}(Y,K)$, where $E_{Y, \lambda+n \mu, K}(\mathfrak{v}_{\xi})=\xi$.
\end{lemma}

\begin{proof} By direct computation,
\begin{equation*}
\begin{split}
\frac{1}{8}c_1(\mathfrak{v}_{\xi})^2-\frac{1}{8}c_1(\mathfrak{h}_{\xi})^2&=\frac{1}{8}(c_1^2(\mathfrak{v}_{\xi})-c_1^2(\mathfrak{v}_{\xi}+PD[D]))\\
&=\frac{1}{2}(-\langle c_1(\mathfrak{v}_{\xi}), [D] \rangle - [D]^2)\\
&=\frac{1}{2p}(\langle c_1(\mathfrak{v}_{\xi}), [\widehat{S}] \rangle+k)   \qquad (\textit{\rm since $[\widehat{S}]$ maps to $-p[D]$, and} \,\, [\widehat{S}]^2=-pk.) \\
&=A_{Y,K}(\xi) \qquad \qquad \qquad \,\,\,\,\,\,\,\, (\textit{\rm by $E_{Y, \lambda+n \mu, K}(\mathfrak{v}_{\xi})=\xi$ and (\ref{HLEA}).})
\end{split} 
\end{equation*}
\end{proof}


\begin{prop}
\label{middlersc=d}
Suppose $K$ is a knot in a rational homology 3-sphere $Y$.  The Alexander grading of the middle relative $\rm Spin^c$ structure $\xi_\mathfrak{s}^0$ is
\begin{equation}\label{AlexGradingofMiddle}
A_{Y,K}(\xi^0_\mathfrak{s})=\frac{1}{2}d(Y, \mathfrak{s})-\frac{1}{2}d(Y, \mathfrak{s}+PD[K]).
\end{equation}
\end{prop}

\begin{proof}
Considering Equation (\ref{VHd1}) for $\xi_\mathfrak{s}^0$, we have 
\begin{align*}
0=V^+_{\xi_\mathfrak{s}^0}(K)-H^+_{\xi_\mathfrak{s}^0}(K)&=\frac{1}{2} d(Y, G_{Y,-K}(\xi_\mathfrak{s}^0)) - \frac{1}{2} d(Y,G_{Y,K}(\xi_\mathfrak{s}^0)) +\dfrac{c_1^2(\mathfrak{v}_{\xi_\mathfrak{s}^0})-c_1^2(\mathfrak{h}_{\xi_\mathfrak{s}^0})}{8}\\
&=\frac{1}{2} d(Y, \mathfrak{s}+PD[K]) - \frac{1}{2} d(Y,\mathfrak{s}) +\dfrac{c_1^2(\mathfrak{v}_{\xi_\mathfrak{s}^0})-c_1^2(\mathfrak{h}_{\xi_\mathfrak{s}^0})}{8}.
\end{align*}
Lemma \ref{vhA} then implies the proposition.
\end{proof}

We now generalize the definition of $\nu^+$-invariant for knots in $S^3$ to rational homology 3-spheres.

\begin{definition}
\label{nudef}
Let $K$ be a knot in a rational homology 3-sphere $Y$. Given $\mathfrak{s} \in {\rm Spin^c}(Y)$, define 
\[\nu^+_\mathfrak{s}(Y,K):={\min}\{A_{Y,K}(\xi)| \xi \in \underline{\rm Spin^c}(Y,K,\mathfrak{s}) \,\, {\rm and} \,\, V^+_{\xi}(K)=0\}\]
and
\[\nu^+(Y,K):=\mathop{\max}\limits_{\mathfrak{s} \in {\rm Spin^c}(Y)} \nu^+_{\mathfrak{s}}(Y,K).\]
\end{definition}


Let $\xi$ be the relative $\rm Spin^c$ structure which supports $\nu^+_\mathfrak{s}(Y,K)$, that is, $A_{Y,K}(\xi)=\nu^+_\mathfrak{s}(Y,K)$ with $\xi \in \underline{\rm Spin^c}(Y,K,\mathfrak{s})$. Then $V^+_{\xi}(K)=0$. If $H^+_{\xi}(K)>0$, then by Lemma \ref{lemmaV-H}, $\nu^+_\mathfrak{s}(Y,K) > A_{Y,K}(\xi^0_{\mathfrak{s}})$; if $H^+_{\xi}(K)=0$, then by Proposition \ref{teuxi0s}, $\nu^+_\mathfrak{s}(Y,K)=A_{Y,K}(\xi^0_{\mathfrak{s}})$. This together with Proposition \ref{middlersc=d} implies (\ref{nugeqd-d}).
\begin{corollary}
\label{nud}
\[\nu^+_\mathfrak{s}(Y,K) \geq A_{Y,K}(\xi^0_{\mathfrak{s}}) = \frac{1}{2}d(Y, \mathfrak{s})-\frac{1}{2}d(Y, \mathfrak{s}+PD[K]).\]

\end{corollary}

\begin{remark}
As a special case, we have the well-known bound $\nu^+(K)\geq 0$ for all knots $K\subset S^3$; in particular, the $0$ on the right-handed side can be interpreted as the Alexander grading of the unique middle relative $\rm Spin^c$ structure in $S^3$.

\end{remark}

\section{Rasmussen conjecture}
\label{Rasmussen conjecture}

In this section, we give a simple new proof of Rasmussen conjecture (Theorem \ref{Rasmussen conjecture NW}), originally proved by Ni and the first author in \cite[Theorem 1.1 and 1.2]{NW}. 


Let
\[ \mathcal{B}_{Y,K}=\{\xi \in \underline{\rm Spin^c}(Y,K) \vert \widehat{HFK}(Y,K,\xi) \neq 0 \},\]

\[\mathcal{B}_{Y,K, \mathfrak{s}}=\{\xi \in \underline{\rm Spin^c}(Y,K, \mathfrak{s}) \vert \widehat{HFK}(Y,K,\xi) \neq 0 \};\]
and let
\[A_{\max}={\max}\{A_{Y,K}(\xi) \vert \xi \in \mathcal{B}_{Y,K}\}, \quad A_{\min}={\min}\{A_{Y,K}(\xi) \vert \xi \in \mathcal{B}_{Y,K}\},\]

\[A^{\mathfrak{s}}_{\max}={\max}\{A_{Y,K}(\xi) \vert \xi \in \mathcal{B}_{Y,K, \mathfrak{s}}\}, \quad A^{\mathfrak{s}}_{\min}={\min}\{A_{Y,K}(\xi) \vert \xi \in \mathcal{B}_{Y,K, \mathfrak{s}}\}.\]

\medskip
\noindent
Suppose $\xi^{\max}, \xi^{\min} \in \mathcal{B}_{Y,K}$ are the relative $\rm Spin^c$ structures satisfying
\[A_{Y,K}(\xi^{\max})=A_{\max}, \quad A_{Y,K}(\xi^{\min})=A_{\min},\]
and $\xi_{\mathfrak{s}}^{\max}, \xi_{\mathfrak{s}}^{\min} \in\mathcal{B}_{Y,K, \mathfrak{s}}$ are the relative $\rm Spin^c$ structures satisfying
\[A_{Y,K}(\xi_{\mathfrak{s}}^{\max})=A^{\mathfrak{s}}_{\max}, \quad A_{Y,K}(\xi_{\mathfrak{s}}^{\min})=A^{\mathfrak{s}}_{\min}.\]
Recall that a knot in an $L$-space is called \textit{Floer simple} if its knot Floer homology is minimal, i.e., 
\[{\rm rank} \,\, \widehat{HFK}(Y,K)={\rm rank} \,\, \widehat{HF}(Y)=|H_1(Y)|.\]
Clearly, $\xi_{\mathfrak{s}}^{\max}=\xi_{\mathfrak{s}}^{\min}=\xi_{\mathfrak{s}}^0$ for Floer simple knots; thus
\begin{equation}
\label{AlexgradingFloersimple}
A^{\mathfrak{s}}_{\max}=A^{\mathfrak{s}}_{\min}=A_{Y,K}(\xi_{\mathfrak{s}}^0).
\end{equation}

\begin{theorem}{\rm \cite[Theorem 1.1]{Ni} \cite[Theorem 2.2]{NW}}
\label{rknotgenus}
Suppose $K$ is a knot in a rational homology 3-sphere $Y$, and $S$ is a minimal genus rational Seifert surface for $K$, then
\[\dfrac{-\chi(S)+\left|[\partial S] \cdot [\mu] \right|}{\left|[\partial S] \cdot [\mu] \right|}=A_{\max}-A_{\min}.\]
\end{theorem}

\begin{proof}[Proof of Theorem \ref{Rasmussen conjecture NW}]
For any $\mathfrak{s} \in {\rm Spin^c}(Y)$, we must have $V^+_{\xi_{\mathfrak{s}}^{\max}}(K)=0$ since the projection $v^+_{\xi_{\mathfrak{s}}^{\max}}: A^+_{\xi_{\mathfrak{s}}^{\max}} \rightarrow B^+_{\xi_{\mathfrak{s}}^{\max}}$ is a quasi-isomorphism. 
Thus, $A^{\mathfrak{s}}_{\max} \geq \nu^+_{\mathfrak{s}}(Y,K)$. Then by Corollary \ref{nud} 
we have 
\[A^{\mathfrak{s}}_{\max} \geq A_{Y,K}(\xi^0_{\mathfrak{s}})=\frac{1}{2} d(Y,\mathfrak{s})- \frac{1}{2}d(Y,\mathfrak{s}+PD[K]).\] 
Hence, 
\[A_{\max}=\mathop{\max}\limits_{\mathfrak{s} \in {\rm Spin^c}(Y)} \left\{ A^{\mathfrak{s}}_{\max} \right\} \geq \mathop{\max}\limits_{\mathfrak{s} \in {\rm Spin^c}(Y)} \left\{ \frac{1}{2} d(Y,\mathfrak{s})- \frac{1}{2}d(Y,\mathfrak{s}+PD[K]) \right\}.\]
In addition, knot Floer homology has the symmetry
\[\widehat{HFK}(Y,K,r) \cong \widehat{HFK}(Y,K,-r),\]
where
\[\widehat{HFK}(Y,K,r)=\mathop{\bigoplus}\limits_{\{\xi \in \underline{\rm Spin^c}(Y,K)|A_{Y,K}(\xi)=r\}} \widehat{HFK}(Y,K,
\xi).\]
See \cite[Section 2.2]{HL} for more details. Thus $A_{\min}=-A_{\max}$, and Theorem \ref{rknotgenus} implies inequality (\ref{NiWuoldresult}) in Theorem \ref{Rasmussen conjecture NW}. Finally, it is clear from (\ref{AlexgradingFloersimple}) that Floer simple knots attain the equality.
\end{proof}

\section{Rational slice genus and adjunction inequality}
\label{4-dimensional rational slice genus}



We proceed to the rational slice genus problems and begin with some homological preliminaries. Suppose $K$ is a knot of order $p$ in a rational homology 3-sphere $Y$. Let $M=Y-N^{\circ}(K)$, where $N^{\circ}(K)$ denotes the interior of the solid torus neighbourhood of $K$. We have the following long exact sequence:
\[0 \rightarrow H_2(M; \mathbb{Q}) \rightarrow H_2(M, \partial M; \mathbb{Q}) \rightarrow H_1(\partial M; \mathbb{Q}) \rightarrow H_1(M; \mathbb{Q}) \rightarrow H_1(M, \partial M; \mathbb{Q}) \rightarrow 0.\]
By Poincar\'e duality and universal coefficient theorem, $H_2(M; \mathbb{Q}) \cong H^1(M, \partial M; \mathbb{Q}) \cong H_1(M, \partial M; \mathbb{Q})$ and $H_2(M, \partial M; \mathbb{Q}) \cong H^1(M; \mathbb{Q}) \cong H_1(M; \mathbb{Q})$, then 
\[0 \rightarrow H_1(M, \partial M; \mathbb{Q}) \rightarrow H_1(M; \mathbb{Q}) \rightarrow H_1(\partial M; \mathbb{Q})
\stackrel{i_{\ast}}{\longrightarrow}  H_1(M; \mathbb{Q}) \rightarrow H_1(M, \partial M; \mathbb{Q}) \rightarrow 0.\]
A simple dimension count shows that ${\rm rank}({\rm ker} (i_{\ast}))={\rm rank}({\rm im}(i_{\ast}))=1$, thus the natural inclusion map with $\mathbb{Z}$-coefficient
\[i_{\ast}: H_1(\partial M; \mathbb{Z}) \rightarrow H_1(M; \mathbb{Z})\]
has ${\rm ker}(i_{\ast}) \cong \mathbb{Z}$ generated by $k\lambda_r$ for some primitive class $\lambda_r \in H_1(\partial M; \mathbb{Z})$ and positive integer $k$. Note that a rational Seifert surface $S$ for $K$ has $[\partial S]=k\lambda_r \in H_1(\partial M; \mathbb{Z})$. 


Next, we consider the Mayer-Vietoris sequence for $Y=M\cup N(K)$:
\[\cdots \longrightarrow  H_1(\partial M ; \mathbb{Z}) \stackrel{i_{\ast} \oplus j_{\ast}}{\longrightarrow} {H_1(M ; \mathbb{Z}) \oplus H_1(N(K) ;\mathbb{Z})}  \stackrel{{\iota}_{\ast}}{\longrightarrow} H_1(Y ; \mathbb{Z}) \longrightarrow 0,\]
As $K$ is a knot of order $p$, the element $(0,p) \in H_1(M; \mathbb{Z}) \oplus H_1(N(K); \mathbb{Z})$ maps to $0 \in H_1(Y ;\mathbb{Z})$. Thus there is a class $\eta \in H_1(\partial M; \mathbb{Z})$ with image $(0,p)$. Since $i_{\ast}(\eta)=0 \in H_1(M; \mathbb{Z})$, $\eta=ak\lambda_r$ for some integer $a \neq 0$. Then $(i_{\ast} \oplus j_{\ast})(k \lambda_r)= (i_{\ast} \oplus j_{\ast})(\frac{\eta}{a})= (0,\frac{p}{a})$, and by exactness, $\iota_{\ast} \circ (i_{\ast} \oplus j_{\ast})(k \lambda_r)=\iota_{\ast}(0,\frac{p}{a})=0 \in H_1(Y; \mathbb{Z})$.  This implies $a=\pm 1$ since $K$ is a knot of order $p$. Therefore $j_{\ast}([\partial S])=j_{\ast}(k \lambda_r)=\pm p[K] \in H_1(N(K); \mathbb{Z})$, which gives $|[\mu] \cdot [\partial S]|=p$.



\medskip

Now, we have seen that all rational Seifert surfaces have the identical boundary curves $k\lambda_r$ completely determined by the homology of $(Y,K)$.  For the rational slice surfaces to be studied in the remainder of this paper, we will consider compact, connected, oriented surfaces $F$ embedded in $Y \times I - N(K)$ with $\partial F=F \cap \partial N(K)$ that satisfy the \textit{Seifert framed condition} $[\partial F]=[\partial S]=k\lambda_r$. 
The rational slice genus for $K$ is defined as the minimal genus of all these Seifert framed rational slice surfaces $F$ of $K$:
\[\|K\|_{Y \times I}^{\partial }=\mathop{\min}\limits_{F} \dfrac{-\chi(F)}{2|[\mu] \cdot [\partial F]|}=\mathop{\min}\limits_{F} \dfrac{-\chi(F)}{2p}.\]

\begin{remark}\label{slicedefinition}
There seems to be no standard definition of 
a rational slice surface in the literature.  In a recent preprint \cite{HR},
Hedden and Raoux defined the rational slice surface with a more relaxed condition on the boundary: their rational slice surface is a compact oriented surface $F$ with boundary, along with a map $\phi: F \rightarrow Y \times I$ satisfying that $\phi|_{F^{\mathrm{o}}}$ is an embedding and $\phi|_{\partial F}$ is a $p$-fold covering of $K \times \{1\}$, where $F^{\mathrm{o}}$ denotes the interior of $F$. In that case, the intersection of $F$ with a solid torus neighborhood of $K$ in $Y \times \{1-\epsilon\}$ for a small $\epsilon$ could be a \textit{satellite of $K$ with any braid pattern}. In contrast, our definition of rational slice surface restricts the intersection to the fixed curves that are parallel to the boundary curves $k\lambda_r$ of the corresponding rational Seifert surfaces of $K$. See \cite[Section 1.2]{HR} for a more detailed comparison of the two definitions. 



\end{remark}


\medskip

We first address the case when the rational longitude $\lambda_r$ for $K$ is a framing, that is, $\lambda_r$ intersects the meridian $\mu$ exactly once.  While most knots do not satisfy this condition, Ozsv\'{a}th-Szab\'{o}'s trick of Morse surgery enables us to convert all knots into this special case. We will elaborate on this point in Section \ref{trick of Morse surgery}.

Consider the 2-handle cobordism $W_{\lambda_r}$ obtained by attaching a 2-handle to $Y \times I$ with $\lambda_r$-framing. As $\lambda_r$ is assumed to be a framing, a Seifert framed rational slice surface $F$ can be capped off by core disks of the 2-handle to a closed surface $\widehat{F}$. Note that the map induced by the cobordism $W_{\lambda_r}$ equipped with a $\rm Spin^c$ structure $\mathfrak{v}$,
\[F^+_{W_{\lambda_r}, \mathfrak{v}}: HF^+(Y, \mathfrak{v}|_Y) \rightarrow HF^+(Y_{\lambda_r}, \mathfrak{v}|_{Y_{\lambda_r}})\]
can be factored though $HF^+(\widehat{F} \times S^1, \mathfrak{v}|_{\widehat{F} \times S^1})$, where we take $\widehat{F} \times S^1$ as the boundary of a tubular neighborhood of $\widehat{F}$ in $W_{\lambda_r}$. The \textit{adjunction inequality} \cite[Thoerem 7.1]{OSpa} for $\widehat{F} \times S^1$ shows \footnote{The case $g(\widehat{F})=0$ can be similarly treated using Heegaard Floer homology with twisted coefficients.} 
\[HF^+(\widehat{F} \times S^1, \mathfrak{v}|_{\widehat{F} \times S^1})=0 \quad {\rm when} \quad |\langle c_1(\mathfrak{v}|_{\widehat{F} \times S^1}), [\widehat{F}]\rangle| \geq 2g(\widehat{F}).\]
In fact,
\[|\langle c_1(\mathfrak{v}|_{\widehat{F} \times S^1}), [\widehat{F}]\rangle|=|\langle c_1(\mathfrak{v}), [\widehat{F}] \rangle|=|\langle c_1(\mathfrak{v}), [\widehat{S}] \rangle|=|\langle c_1(\mathfrak{v}|_{Y_{\lambda_r}}), [\widehat{S}] \rangle|.\]
So we conclude that when $|\langle c_1(\mathfrak{v}|_{Y_{\lambda_r}}), [\widehat{S}] \rangle| \geq 2g(\widehat{F})$, the map $F^+_{W_{\lambda_r}, \mathfrak{v}}$ is \textit{trivial}. 

\medskip
In the next section, we will identify this cobordism map with an inclusion map in the mapping cone formulation, which can be, in turn, understood in terms of the $V$ and $H$ invariants that were introduced in Section \ref{Spinc structures and Alexander grading} . To simplify the analysis, we will also introduce Rasmussen's notation.

\bigskip

\section{Mapping cone formula and Rasmussen notation}
\label{Mapping cone formula and Rasmussen notation}


\subsection{Mapping cone formula}
\label{Mapping cone formula}

The Heegaard Floer mapping cone formula is a tool to compute the Heegaard Floer homology of surgery along a knot, and it describes the cobordism map for a 2-handle attachment in terms of the knot Floer complex. Here, we use the mapping cone formula for $HF^-$. In order to apply the minus version mapping cone formula, we must work with the completion of $HF^-$.


Given a 3-manifold $Y$ with a $\rm Spin^c$ structure $\mathfrak{s}$, we define $\boldsymbol{CF^-}(Y, \mathfrak{s})=CF^-(Y, \mathfrak{s}) \otimes_{\mathbb{F}[U]} \mathbb{F}[[U]]$. The homology of $\boldsymbol{CF^-}(Y, \mathfrak{s})$ is defined as $\boldsymbol{HF^-}(Y, \mathfrak{s})$. The cobordism maps $F^-_{W,\mathfrak{t}}: HF^-(Y_0, \mathfrak{s}_0) \rightarrow HF^-(Y_1, \mathfrak{s}_1)$ also have analogues in completed setting, denoted by 
\[\boldsymbol{F^-}_{W,\mathfrak{t}}: \boldsymbol{HF^-}(Y_0, \mathfrak{s}_0) \rightarrow \boldsymbol{HF^-}(Y_1, \mathfrak{s}_1).\]
See \cite{MO} for more details.


Now we review the mapping cone formula for $\boldsymbol{HF^-}$ in \cite{MO}. 
Let $K$ be an oriented knot in a rational homology 3-sphere $Y$. Let $\mu$ be the meridian of $K$ and $\lambda$ a framing which naturally inherits an orientation from $K$. 
Similar to the plus version introduced in Section \ref{Spinc structures and Alexander grading}, we define
\[A^-_{\xi}(K)=C_{\xi}\{\max\{i,j\} \leq 0\} \,\,\, {\rm and} \,\,\, B^-_{\xi}(K)=C_{\xi}\{i \leq 0\}.\]
Basically, the complexes 
\begin{align*}
C_{\xi}\{i \leq 0\}&\cong CF^-(Y, G_{Y,K}(\xi))\cong B^-_{\xi}(K), \\
C_{\xi}\{j \leq 0\}&
\cong CF^-(Y,G_{Y,-K}(\xi))\cong CF^-(Y,G_{Y,K}(\xi+PD[\lambda]))\cong B^-_{\xi +PD[\lambda]}(K),
\end{align*}
while $A^-_{\xi}(K)$ is quasi-isomorphic to the complex $CF^-$ of a large surgery $Y_{\lambda+n\mu}(K)$ for $n \gg 0$ in a certain $\rm Spin^c$ structure. 
There are two natural inclusion maps
\[v^-_{\xi}: A^-_{\xi}(K) \rightarrow B^-_{\xi}(K), \quad h^-_{\xi}: A^-_{\xi}(K) \rightarrow B^-_{\xi+PD[\lambda]}(K),\]
which induce maps
\[v^-_{\xi,\ast}: H_{\ast}(A^-_{\xi}(K)) \rightarrow H_{\ast}(B^-_{\xi}(K)), \quad h^-_{\xi,\ast}: H_{\ast}(A^-_{\xi}(K)) \rightarrow H_{\ast}(B^-_{\xi+PD[\lambda]}(K)).\]
As both $H_{\ast}(B^-_{\xi}(K))$ and $H_{\ast}(A^-_{\xi}(K))$ are (non-canonically) isomorphic to the direct sum of $\mathbb{F}[U]$ and a finite-dimensional $U$-torsion module, 
we can define
\[V^-_{\xi}(K)={\rm rank}(\mathbb{F}[U] / ({\rm Im}(v^-_{\xi, \ast})), \quad H^-_{\xi}(K)={\rm rank}(\mathbb{F}[U] / ({\rm Im}(h^-_{\xi, \ast})).\]

\begin{remark} Hom and Lidman's argument in \cite[Lemma 2.6]{HomL} can be applied here to show that $V^-_{\xi}(K)$ and $H^-_{\xi}(K)$ coincide with the respective plus version of the invariant $V^+_{\xi}(K)$ and $H^+_{\xi}(K)$ in Section \ref{Spinc structures and Alexander grading}. 
\end{remark}

When we take the completion of the minus theory, we obtain the complexes $\boldsymbol{A}^-_{\xi}$ and $\boldsymbol{B}^-_{\xi}$, as well as maps $\boldsymbol{v}^-_{\xi}$ and $\boldsymbol{h}^-_{\xi}$. 
Given any $\mathfrak{s} \in {\rm Spin^c}(Y_{\lambda}(K))$, let
\begin{align*}
\boldsymbol{A}^-_{\mathfrak{s}}&=\mathop{\prod}_{\{\xi \in \underline{\rm Spin^c}(Y_{\lambda}(K),K_{\lambda})| G_{Y_{\lambda}(K),K_{\lambda}}(\xi)=\mathfrak{s}\}} \boldsymbol{A}^-_{\xi}, \\
\boldsymbol{B}^-_{\mathfrak{s}}&=\mathop{\prod}_{\{\xi \in \underline{\rm Spin^c}(Y_{\lambda}(K),K_{\lambda})| G_{Y_{\lambda}(K),K_{\lambda}}(\xi)=\mathfrak{s}\}} \boldsymbol{B}^-_{\xi}, 
\end{align*}
where $K_{\lambda}$ represents the oriented dual knot of the knot $K$ in the surgered manifold $Y_{\lambda}(K)$, and $G_{Y_{\lambda}(K),K_{\lambda}}: \underline{\rm Spin^c}(Y_{\lambda}(K),K_{\lambda}) \rightarrow {\rm Spin^c}(Y_{\lambda}(K))$. Note that $\underline{\rm Spin^c}(Y,K)=\underline{\rm Spin^c}(Y_{\lambda}(K),K_{\lambda})$, since they both represent the set of the relative ${\rm Spin^c}$ structures on the knot complement $Y-K=Y_{\lambda}(K)-K_{\lambda}$. 
Define
\[\boldsymbol{D}^-_{\mathfrak{s}}: \boldsymbol{A}^-_{\mathfrak{s}} \rightarrow \boldsymbol{B}^-_{\mathfrak{s}}, \quad
(\xi, a) \mapsto (\xi, \boldsymbol{v}^-_{\xi}(a))+(\xi+PD[\lambda], \boldsymbol{h}^-_{\xi}(a)).\]

\begin{theorem}{\rm \cite[Theorem 1.1 and 14.3]{MO}}
\label{mappingconethm}
For any $\mathfrak{s} \in {\rm Spin^c}(Y_{\lambda}(K))$, the Heegaard Floer homology $\boldsymbol{HF}^-(Y_{\lambda}(K),\mathfrak{s})$ is isomorphic to the homology of the mapping cone $\boldsymbol{X}^-_{\mathfrak{s}}$ of $\boldsymbol{D}^-_{\mathfrak{s}}$. Moreover, under this isomorphism,  for any $\xi \in \underline{\rm Spin^c}(Y_{\lambda}(K),K_{\lambda})$ with $G_{Y_{\lambda}(K),K_{\lambda}}(\xi)=\mathfrak{s}$, the inclusion map 
\[H_{\ast}(\boldsymbol{B}^-_{\xi}) \rightarrow H_{\ast}(\boldsymbol{X}^-_{\mathfrak{s}})\]
is identified with the map
\[\boldsymbol{HF}^-(Y, G_{Y,K}(\xi)) \rightarrow \boldsymbol{HF}^-(Y_{\lambda}(K),\mathfrak{s}),\]
induced by the natural 2-handle cobordism $W_{\lambda}: Y \rightarrow Y_{\lambda}$ endowed with the corresponding $\rm Spin^c$ structure.
\end{theorem}

Denote
\[\mathfrak{A}^-_{\xi}=H_{\ast}(\boldsymbol{A}^-_{\xi}) \,\, ({\rm resp.} \,\, \mathfrak{B}^-_{\xi}=H_{\ast}(\boldsymbol{B}^-_{\xi})), \qquad \mathfrak{A}^-_{\mathfrak{s}}=H_{\ast}(\boldsymbol{A}^-_{\mathfrak{s}}) \,\, ({\rm resp.} \,\, \mathfrak{B}^-_{\mathfrak{s}}=H_{\ast}(\boldsymbol{B}^-_{\mathfrak{s}})). \]
Let
\[\mathfrak{v}^-_{\xi}: \mathfrak{A}^-_{\xi} \rightarrow \mathfrak{B}^-_{\xi}, \qquad \mathfrak{h}^-_{\xi}: \mathfrak{A}^-_{\xi} \rightarrow \mathfrak{B}^-_{\xi+PD[\lambda]}\]
be the maps induced on homology by $\boldsymbol{v}^-_{\xi}$ and $\boldsymbol{h}^-_{\xi}$ respectively, and 
let
\[\mathfrak{D}^-_{\mathfrak{s}}: \mathfrak{A}^-_{\mathfrak{s}} \rightarrow \mathfrak{B}^-_{\mathfrak{s}}\]
be the map induced on homology by $\boldsymbol{D}^-_{\mathfrak{s}}$. Theorem \ref{mappingconethm} implies the exact triangle 
\begin{equation}
\xymatrix{
\mathfrak{A}^-_{\mathfrak{s}} \ar[r]^{\mathfrak{D}^-_{\mathfrak{s}}} & \mathfrak{B}^-_{\mathfrak{s}} \ar[d]^{{incl}_{\ast}} \\
 & 
 \boldsymbol{HF}^-(Y_{\lambda}(K),\mathfrak{s})
 \ar[lu]^{{proj}_{\ast}}.}
 \label{mappingconeexacttriangle}
\end{equation}


\subsection{Mapping cone for surgeries along rational longitudes}
Suppose $K$ is a knot of order $p$ in a rational homology 3-sphere $Y$.  Let $\lambda_r$ be the rational longitude of $K$ that is assumed to be a framing. In this section, we study $\lambda_r$-surgery along $K$. Unlike the familiar linear mapping cone, we will encounter a \textit{circular} mapping cone since $p \cdot [\lambda_r]=[\partial S]=0 \in H_1(Y-K)$. 
See Figure \ref{circularmappingcone} for an illustration.

\begin{figure}[!h]
\centering
\begin{tikzpicture}
\node (1) at(0,4){\small $\mathfrak{A}^-_{\xi}$};
\node (2) at(1,1){\small $\mathfrak{B}^-_{\xi+PD[\lambda_r]}$};
\node (3) at(4.2,1){\small $\mathfrak{A}^-_{\xi+PD[\lambda_r]}$};
\node (4) at(1.8,-0.9){\small $\mathfrak{B}^-_{\xi+2PD[\lambda_r]}$};
\node (5) at(2.8,-3.8){\small $\mathfrak{A}^-_{\xi+2PD[\lambda_r]}$};
\node (6) at(-2.8,-3.8){\small $\mathfrak{A}^-_{\xi+(p-2)PD[\lambda_r]}$};
\node (7) at(-1.8,-0.9){\small $\mathfrak{B}^-_{\xi+(p-1)PD[\lambda_r]}$};
\node (8) at(-1,1){\small $\mathfrak{B}^-_{\xi}$};
\node (9) at(-4.2,1){\small $\mathfrak{A}^-_{\xi+(p-1)PD[\lambda_r]}$};
\node (10) at(0,-1.8){$\dots$};
\node (11) at(0,-3.8){$\dots$};

\draw[->] (1) --node[right]{\scriptsize $\mathfrak{h}^-_{\xi}$} (2);
\draw[->] (3) --node[above]{\scriptsize $\mathfrak{v}^-_{\xi+PD[\lambda_r]}$}(2);
\draw[->] (3) --node[sloped,below]{\scriptsize $\mathfrak{h}^-_{\xi+PD[\lambda_r]}$}(4);
\draw[->] (5) --node[sloped,above]{\scriptsize $\mathfrak{v}^-_{\xi+2PD[\lambda_r]}$}(4);
\draw[->] (5) --node[sloped,below]{\scriptsize $\mathfrak{h}^-_{\xi+2PD[\lambda_r]}$}(10);
\draw[->] (6) --node[sloped,below]{\scriptsize $\mathfrak{v}^-_{\xi+(p-2)PD[\lambda_r]}$}(10);
\draw[->] (6) --node[sloped,above]{\scriptsize $\mathfrak{h}^-_{\xi+(p-2)PD[\lambda_r]}$} (7);
\draw[->] (9) --node[sloped,above]{\quad \quad \,\, \scriptsize $\mathfrak{v}^-_{\xi+(p-1)PD[\lambda_r]}$}(7);
\draw[->] (9) --node[above]{\scriptsize $\mathfrak{h}^-_{\xi+(p-1)PD[\lambda_r]}$}(8);
\draw[->] (1) --node[left]{\scriptsize $\mathfrak{v}^-_{\xi}$}(8);
\end{tikzpicture}
\caption{The circular mapping cone for $\lambda_r$-surgery.}
\label{circularmappingcone}
\end{figure}

Note that the $p$ relative ${\rm Spin^c}$ structures $\xi + i \cdot PD[\lambda_r]$ with $i \in \mathbb{Z}_p$ appeared in the circular mapping cone all have the \textit{same} Alexander grading:
\begin{align*}
A_{Y,K}(\xi + i \cdot PD[\lambda_r])&=\dfrac{\langle c_1(\xi + i \cdot PD[\lambda_r]), [S] \rangle + [\mu] \cdot [S]}{2[\mu] \cdot [S]}\\
&=\dfrac{\langle c_1(\xi), [S] \rangle + 2i[\lambda_r] \cdot [S] + [\mu] \cdot [S]}{2[\mu] \cdot [S]}\\
&=\dfrac{\langle c_1(\xi), [S] \rangle + p}{2p},
\end{align*}
since $[\lambda_r] \cdot [S]=0$. Indeed, we have:
\begin{lemma}
\label{AandC1S}
For any $i \in \mathbb{Z}_p$, 
\[A_{Y,K}(\xi + i \cdot PD[\lambda_r])=\frac{1}{2p}\langle c_1(\mathfrak{s}), [\widehat{S}] \rangle,\] 
where $\mathfrak{s}=G_{Y_{\lambda_r},K_{\lambda_r}}(\xi + i \cdot PD[\lambda_r]) \in {\rm Spin^c}(Y_{\lambda_r})$. 
\end{lemma}

\begin{proof}
Recall that the map $G_{Y,K}$ is specified as gluing 
Turaev's distinguished Euler structure $w$ on $S^1 \times D^2$ to a relative $\rm Spin^c$ structure on $Y - N^{\circ}(K)$. In fact, $w$ satisfies that $c_1(w)=PD[c]$, where $c$ denotes the core of $S^1 \times D^2$. Let $\bar{D}$ denote the meridian disk of $S^1 \times D^2$, then $\widehat{S}=S \cup p \bar{D}$. Thus for any $i \in \mathbb{Z}_p$,
\begin{align*}
\frac{\langle c_1(\mathfrak{s}), [\widehat{S}] \rangle}{2p}&=\frac{\langle c_1((\xi+ i \cdot PD[\lambda_r]) \cup w), [\widehat{S}] \rangle}{2p}\\
&=\frac{\langle c_1(\xi+ i \cdot PD[\lambda_r]), [S] \rangle + \langle c_1(w), p[\bar{D}] \rangle}{2p}\\
&=\frac{\langle c_1(\xi+ i \cdot PD[\lambda_r]), [S] \rangle +p}{2p}\\
&=A_{Y,K}(\xi + i \cdot PD[\lambda_r]),
\end{align*}
where the $\cup$ notation in the first equality denotes a gluing map of relative $\rm Spin^c$ structures. See \cite[Chapter VI.1.1]{Tur}.
\end{proof}

Recall now that in Section \ref{4-dimensional rational slice genus}, we proved the map 
\[F^+_{W_{\lambda_r}, \mathfrak{v}}: HF^+(Y, \mathfrak{v}|_Y) \rightarrow HF^+(Y_{\lambda_r}, \mathfrak{v}|_{Y_{\lambda_r}})\]
is trivial when $|\langle c_1(\mathfrak{v}|_{Y_{\lambda_r}}), [\widehat{S}] \rangle| \geq 2g(\widehat{F})$. As 
\[\boldsymbol{HF}^-(Y, \mathfrak{t}) \cong HF^+(Y, \mathfrak{t})\]
for any 3-manifold equipped with a non-torsion $\rm Spin^c$ structure $\mathfrak{t}$ \cite[Section 2]{MO}, the map 
\[\boldsymbol{F}^-_{W_{\lambda_r}, \mathfrak{v}}: \boldsymbol{HF}^-(Y, \mathfrak{v}|_Y) \rightarrow \boldsymbol{HF}^-(Y_{\lambda_r},\mathfrak{v}|_{Y_{\lambda_r}})\]
must also be trivial when $|\langle c_1(\mathfrak{v}|_{Y_{\lambda_r}}), [\widehat{S}] \rangle| \geq 2g(\widehat{F})$, since it factors through $\boldsymbol{HF}^-(\widehat{F} \times S^1, \mathfrak{v}|_{\widehat{F} \times S^1}) \cong HF^+(\widehat{F} \times S^1, \mathfrak{v}|_{\widehat{F} \times S^1})$. Theorem \ref{mappingconethm} then implies that for all $i \in \mathbb{Z}_p$, the map 
\[\mathfrak{B}^-_{\xi+i \cdot PD[\lambda_r]} \rightarrow H_{\ast}(\boldsymbol{X}^-_{\mathfrak{s}})\]
is trivial if $|\langle c_1(\mathfrak{s}), [\widehat{S}] \rangle| \geq 2g(\widehat{F})$, where $\mathfrak{s}=G_{Y_{\lambda_r},K_{\lambda_r}}(\xi + i \cdot PD[\lambda_r])$. 
Consequently, 
\[\mathfrak{D}^-_{\mathfrak{s}}: \mathfrak{A}^-_{\mathfrak{s}} \rightarrow \mathfrak{B}^-_{\mathfrak{s}}\]
is surjective by Exact triangle (\ref{mappingconeexacttriangle}). Combining with Lemma \ref{AandC1S}, we obtain the following mapping cone version of the adjunction inequality.  

\begin{proposition}
\label{mappingconeAdjunction}
Suppose $K$ is a knot of order $p$ in a rational homology 3-sphere $Y$ whose rational longitude $\lambda_r$ is a framing.
If a relative $\rm Spin^c$ structure $\xi \in \underline{\rm Spin^c}(Y,K)$ satisfies  
$$A_{Y,K}(\xi+ i \cdot PD[\lambda_r]) \geq \frac{g(\widehat{F})}{p}$$  for a Seifert framed rational slice surface $F$ and any $i \in \mathbb{Z}_p$, then the map
$\mathfrak{D}^-_{\mathfrak{s}}: \mathfrak{A}^-_{\mathfrak{s}} \rightarrow \mathfrak{B}^-_{\mathfrak{s}}$ is  surjective, where $\mathfrak{s}=G_{Y_{\lambda_r},K_{\lambda_r}}(\xi + i \cdot PD[\lambda_r])$. 
\end{proposition}

\subsection{Rasmussen's notation and non-surjective label pattern}

In view of Proposition \ref{mappingconeAdjunction}, we want to analyze when $\mathfrak{D}^-_{\mathfrak{s}}: \mathfrak{A}^-_{\mathfrak{s}} \rightarrow \mathfrak{B}^-_{\mathfrak{s}}$ must not be surjective. To this purpose, we generalize Rasmussen's notation, first appeared in \cite{RasNotation}.
\begin{definition}
Given a knot $K$ in a rational homology 3-sphere $Y$ and any relative ${\rm Spin^c}$ structure $\xi \in \underline{\rm Spin^c}(Y,K)$, the complex $\mathfrak{A}^-_{\xi}$ is labeled with one of the following 4 types of symbols depending on $V^-_{\xi}(K)$ and $H^-_{\xi}(K)$:
\begin{itemize}
\item[(1)] $\mathfrak{A}^-_{\xi}$ is of type $\circ$ if $V^-_{\xi}(K)=H^-_{\xi}(K)=0$;
\item[(2)] $\mathfrak{A}^-_\xi$ is of type $+$ if $V^-_{\xi}(K)=0, H^-_{\xi}(K)>0$;
\item[(3)] $\mathfrak{A}^-_\xi$ is of type $-$ if $V^-_{\xi}(K)>0, H^-_{\xi}(K)=0$;
\item[(4)] $\mathfrak{A}^-_\xi$ is of type $\ast$ if $V^-_{\xi}(K)>0, H^-_{\xi}(K)>0$.
\end{itemize}

For simplicity, we sometimes also write $\mathfrak{A}^-_{\xi}=\circ, +, -$ or $\ast$ to indicate the respective types.
\end{definition}

Given any $\rm Spin^c$ structure $\mathfrak{t} \in {\rm Spin^c}(Y)$, we denote the set of all $\mathfrak{A}^-_{\xi+i \cdot PD[\mu]}$ for $K$ with $\xi \in \underline{\rm Spin^c}(Y,K, \mathfrak{t})$ and $i \in \mathbb{Z}$ by $\mathfrak{G}^-_{\mathfrak{t}}(Y,K)$. Due to the monotonicity (\ref{Vd}), (\ref{Hi}), Lemma \ref{lemmaV-H} and Proposition \ref{teuxi0s}, $\mathfrak{G}^-_{\mathfrak{t}}(Y,K)$ must belong to one of the following two types in Rasmussen's notation: 

\begin{itemize}
\item [(\romannumeral1)] $\mathfrak{G}^-_{\mathfrak{t}}(Y,K)$ contains only one $\circ$ and all other elements are $+$'s and $-$'s; 
\item[(\romannumeral2)] $\mathfrak{G}^-_{\mathfrak{t}}(Y,K)$ contains several $\ast$'s and all other elements are $+$'s and $-$'s.
\end{itemize}

If we put $\mathfrak{G}^-_{\mathfrak{t}}(Y,K)$ in a column, then the two types are depicted in Table \ref{type12}.


\begin{table}
\centering
 \renewcommand\arraystretch{1.3} 
\setlength{\tabcolsep}{5mm}{
\begin{tabular}{ccc}
$\vdots$ & $\vdots$ & $\vdots$\\
$\mathfrak{A}^-_{\xi^0_{\mathfrak{t}}+3 \cdot PD[\mu]}$ & $+$ & $+$ \\
$\mathfrak{A}^-_{\xi^0_{\mathfrak{t}}+2 \cdot PD[\mu]}$ & $+$ & $+$\\
$\mathfrak{A}^-_{\xi^0_{\mathfrak{t}}+PD[\mu]}$ & $+$ & $\ast$\\
$\mathfrak{A}^-_{\xi^0_{\mathfrak{t}}}$ & $\circ$ & $\ast$ \\
$\mathfrak{A}^-_{\xi^0_{\mathfrak{t}}-PD[\mu]}$ & $-$ & $\ast$ \\
$\mathfrak{A}^-_{\xi^0_{\mathfrak{t}}-2 \cdot PD[\mu]}$ & $-$ & $-$ \\
$\mathfrak{A}^-_{\xi^0_{\mathfrak{t}}-3 \cdot PD[\mu]}$ & $-$ & $-$ \\
$\vdots$ & $\vdots$ & $\vdots$\\
& {\tiny Type (\romannumeral1)} & {\tiny Type (\romannumeral2)}
\end{tabular}}
\vspace{1.5mm}
\caption{Two types $\mathfrak{G}^-_{\mathfrak{t}}(Y,K)$}
\label{type12}
\end{table}

\begin{remark}
\label{FloersimpleType}
Floer simple knots are of type (\romannumeral1) for any $\rm Spin^c$ structure.
\end{remark}

We use Rasmussen's notation to represent the circular mapping cone $\mathfrak{D}^-_{\mathfrak{s}}: \mathfrak{A}^-_{\mathfrak{s}} \rightarrow \mathfrak{B}^-_{\mathfrak{s}}$ in Figure \ref{Rasmussen's notation}: Here, the outer circle of the diagram represents $\mathfrak{A}^-_\xi$, while the inner circle represents $\mathfrak{B}^-_\xi$. Each $\mathfrak{B}^-_\xi$ is represented by a filled circle. Maps with $V^-_{\xi}=0$ or $H^-_{\xi}=0$ are indicated by arrows, and maps with $V^-_{\xi}>0$ or $H^-_{\xi}>0$ are omitted. Equivalently, these arrows indicate whether the generator $1$ of the tower $\mathbb{F}[[U]]$ in each $\mathfrak{B}^-_\xi$ lies in the image of the corresponding maps $\mathfrak{v}^-_{\xi}$ or $\mathfrak{h}^-_{\xi}$.  

\begin{figure}[!h]
\label{Rasmussen's notation}
\centering
\begin{tikzpicture}

\node (1) at(2,0){$\circ$};
\node (2) at(1.73,1){$\circ$};
\node (3) at(1,1.73){$\circ$};
\node (4) at(0,2){$+$};
\node (5) at(-1,1.73){$\circ$};
\node (6) at(-1.73,1){$\circ$};
\node (7) at(-2,0){$-$};
\node (8) at(-1.73,-1){$\circ$};
\node (9) at(-1,-1.73){$\circ$};
\node (10) at(0,-2){$\ast$};
\node (11) at(1,-1.73){$\circ$};
\node (12) at(1.73,-1){$\ast$};
\node (13) at(0.966,0.26){$\bullet$};
\node (14) at(0.7,0.7){$\bullet$};
\node (15) at(0.26,0.966){$\bullet$};
\node (16) at(-0.26,0.966){$\bullet$};
\node (17) at(-0.7,0.7){$\bullet$};
\node (18) at(-0.966,0.26){$\bullet$};
\node (19) at(-0.966,-0.26){$\bullet$};
\node (20) at(-0.7,-0.7){$\bullet$};
\node (21) at(-0.26,-0.966){$\bullet$};
\node (22) at(0.26,-0.966){$\bullet$};
\node (23) at(0.7,-0.7){$\bullet$};
\node (24) at(0.966,-0.26){$\bullet$};

\draw[->] (1) --(13);
\draw[->] (2) --(13);
\draw[->] (2) --(14);
\draw[->] (3) --(14);
\draw[->] (3) --(15);
\draw[->] (4) --(16);
\draw[->] (5) --(16);
\draw[->] (5) --(17);
\draw[->] (6) --(17);
\draw[->] (6) --(18);
\draw[->] (7) --(18);
\draw[->] (8) --(19);
\draw[->] (8) --(20);
\draw[->] (9) --(20);
\draw[->] (9) --(21);
\draw[->] (11) --(22);
\draw[->] (11) --(23);
\draw[->] (1) --(24);
\end{tikzpicture}
\caption{Rasmussen's notation}
\label{Rasmussen's notation}
\end{figure}
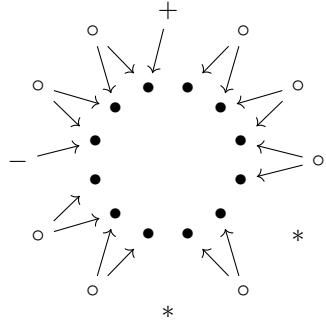

We draw the circular mapping cones in a clockwise order. We use an interval $[a,b]$ to represent a local segment of circular mapping cones, where $a$ and $b$ are labeled with a $+$, $-$ or $\ast$, and all the elements in between are $\circ$. See Figure \ref{Non-surjective interval}. There might be more than one or no $\circ$ in between, but for convenience, we put one $\circ$ in every diagram in Figure \ref{Non-surjective interval}.

\begin{lemma}
\label{nonsurint}
If the labeled diagram in Rasmussen's notation for a map $\mathfrak{D}^-_{\mathfrak{s}}: \mathfrak{A}^-_{\mathfrak{s}} \rightarrow \mathfrak{B}^-_{\mathfrak{s}}$ contains one of the 4 types of intervals $[\ast, \ast]$, $[+,-]$, $[+, \ast]$ and $[\ast, -]$, then $\mathfrak{D}^-_{\mathfrak{s}}$ is not surjective.
\end{lemma}

\begin{proof}
For the interval $[\ast, \ast]$, the element $(0,0, \dots, 0, 1)$ in $\mathfrak{B}^-_{\mathfrak{s}}$ is not in the image of $\mathfrak{D}^-_{\mathfrak{s}}$, where $(0,0, \dots, 0, 1)$ represents the element whose all components are $0$ except the last one under the right-handed $\ast$ which equals the generator $1$ of the tower $\mathbb{F}[[U]]$.  Thus, a circular mapping cone containing a $[\ast,\ast]$-interval is not surjective. Similarly, the element $(0,0, \dots, 0, 1)$ in $\mathfrak{B}^-_{\mathfrak{s}}$ is not in the image of $\mathfrak{D}^-_{\mathfrak{s}}: \mathfrak{A}^-_{\mathfrak{s}} \rightarrow \mathfrak{B}^-_{\mathfrak{s}}$ with $\mathfrak{A}^-_{\mathfrak{s}}$ containing an interval of type $[+,-]$, $[+, \ast]$ or $[\ast, -]$, where the $1$ component is under the right-handed $-$ or $\ast$.
\end{proof}

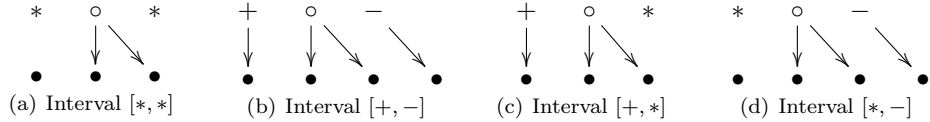
\begin{figure}[H]
\centering
\subfigure[Interval $\lbrack \ast , \ast \rbrack$]{
\begin{minipage}[t]{0.2\textwidth}
						\centering
						\begin{displaymath}
						\xymatrixrowsep{5mm}
						\xymatrixcolsep{4mm}
						\xymatrix{
							\ast & \circ \ar[d] \ar[rd] & \ast\\
							\bullet & \bullet & \bullet }
						\end{displaymath}
					\end{minipage}}
\subfigure[Interval $\lbrack +,- \rbrack$]{
\begin{minipage}[t]{0.2\textwidth}
			\centering
			\begin{displaymath}
			\xymatrixrowsep{5mm}
			\xymatrixcolsep{4mm}
			\xymatrix{
				+ \ar[d] & \circ \ar[d] \ar[rd] & - \ar[rd]\\
				\bullet & \bullet & \bullet& \bullet}
			\end{displaymath}
		\end{minipage}}
\subfigure[Interval $\lbrack + ,\ast \rbrack$]{
\begin{minipage}[t]{0.2\textwidth}
				\centering
				\begin{displaymath}
				\xymatrixrowsep{5mm}
				\xymatrixcolsep{4mm}
				\xymatrix{
					+ \ar[d] & \circ \ar[d] \ar[rd] & \ast  \\
					\bullet & \bullet & \bullet }
				\end{displaymath}
			\end{minipage}}
\subfigure[Interval $\lbrack \ast, - \rbrack$]{
		\begin{minipage}[t]{0.2\textwidth}
					\centering
					\begin{displaymath}
					\xymatrixrowsep{5mm}
					\xymatrixcolsep{4mm}
					\xymatrix{
						\ast  & \circ \ar[d] \ar[rd] & - \ar[rd] &\\
						\bullet & \bullet & \bullet & \bullet}
					\end{displaymath}
				\end{minipage}}
\caption{Non-surjective interval}\label{Non-surjective interval}
\end{figure}

\begin{lemma}
\label{notsurpat}
The map $\mathfrak{D}^-_{\mathfrak{s}}: \mathfrak{A}^-_{\mathfrak{s}} \rightarrow \mathfrak{B}^-_{\mathfrak{s}}$ is not surjective if $\mathfrak{A}^-_{\mathfrak{s}}$ contains a $\ast$.
\end{lemma}

\begin{proof}
Suppose there exists a $\ast$ in $\mathfrak{A}^-_{\mathfrak{s}}$. We start from that $\ast$ and go clockwise along the outer circle: \\

Case 1: If the first non-$\circ$ element we meet is either $\ast$ or $-$, then we get an interval of type $[\ast, \ast]$ or $[\ast, -]$, and Lemma \ref{nonsurint} implies that $\mathfrak{D}^-_{\mathfrak{s}}$ is not surjective. In particular, in the case where all elements are $\circ$'s before we return to the original $\ast$ on the circle, then we have a special $[\ast,\ast]$-interval for which the argument in Lemma \ref{nonsurint} still applies. \\

Case 2: If the first non-$\circ$ element we meet is a $+$, then we can keep going until we meet a $-$ or a $\ast$, which will definitely happen since we will eventually return to the original $\ast$ if we do not meet any $-$. Then we will have an interval of type $[+,-]$ or $[+,\ast]$, and Lemma \ref{nonsurint} implies that $\mathfrak{D}^-_{\mathfrak{s}}$ is not surjective. 
\end{proof}

\section{The rational slice genus bound for knots whose rational longitude is a framing}
\label{The rational slice genus bound for knots whose rational longitude is a framing}

In this section, we show our slice genus bound for knots whose rational longitude is a framing. Let us first recall some useful results in Heegaard Floer theory. 


\subsection{Symmetries in knot Floer homology}

Let $K$ be a knot in a rational homology 3-sphere $Y$. Reversing both the orientations of $Y$ and $K$, we obtain the mirror $-K\subset -Y$. 
Given any $\rm Spin^c$ structure $\mathfrak{s} \in {\rm Spin^c}(Y) \cong {\rm Spin^c}(-Y)$, we have the filtered chain homotopy equivalence 
\begin{equation}
CFK^{\infty}(-Y,-K, J\mathfrak{s}) \simeq CFK^{\infty}(Y,K,\mathfrak{s})^{\ast},
\end{equation}
where $CFK^{\infty}(Y,K,\mathfrak{s})^{\ast}$ denotes the dual complex, i.e., ${\rm Hom}_{\mathbb{F}[U,U^{-1}]}(CFK^{\infty}(Y,K,\mathfrak{s}),\mathbb{F}[U,U^{-1}])$, and $J$ is the conjugation on ${\rm Spin^c}(Y)$. More explicitly, we can obtain $CFK^{\infty}(Y,K,\mathfrak{s})^{\ast}$ by rotating the complex $CFK^{\infty}(Y,K,\mathfrak{s})$ by $180^{\circ}$ and then reversing the directions of all of the arrows. 
See \cite[Section 3.5]{OSk} for more details. 

\medskip
Our next statement will simultaneously involve $V^-_{\xi}$ invariants for both $(Y,K)$ and its mirror $(-Y,-K)$. To avoid potential confusion, we will write $V^-_{\xi}(Y,K)$ and $V^-_{\xi}(-Y,-K)$ to indicate which knot we are considering.  In \cite[Proposition 3.11]{Hom}, Hom proved the result below for knots in $S^3$, but the argument immediately generalizes to the situation considered here.
\begin{prop}{\rm \cite[Proposition 3.11]{Hom}}
\label{V=V=0}
Given any pair of conjugated $\rm Spin^c$ structures $\mathfrak{s}, J\mathfrak{s} \in {\rm Spin^c}(Y) \cong {\rm Spin^c}(-Y)$, let $\xi^0_{\mathfrak{s}}$ and $\xi^0_{J\mathfrak{s}}$ be the middle relative $\rm Spin^c$ structures associated to $\mathfrak{s}$ and $J\mathfrak{s}$ respectively. If 
\[V^-_{\xi^0_{\mathfrak{s}}}(Y,K)=V^-_{\xi^0_{J\mathfrak{s}}}(-Y,-K)=0,\] 
then there exists a basis for $CFK^{\infty}(Y,K,\mathfrak{s})$ with a basis element $x$ which generates the homology $HFK^{\infty}(Y,K,\mathfrak{s})$ and splits off as a direct summand of $CFK^{\infty}(Y,K,\mathfrak{s})$, where $x$ is supported in the relative $\rm Spin^c$ structure $\xi^0_{\mathfrak{s}}$. In other words,
\[CFK^{\infty}(Y,K,\mathfrak{s}) \simeq CFK^{\infty}(S^3,U) \oplus A,\]
where $U$ denotes the unknot in $S^3$ and $A$ is some acyclic complex.
\end{prop}

\medskip

In \cite[Section 3.5]{OSk}, Ozsv\'{a}th and Szab\'{o} exhibited another symmetry between $CFK^{\infty}(Y,K,\mathfrak{\xi})$ and $CFK^{\infty}(Y,K,\tilde{J}\xi)$. 
See also \cite[Section 3]{NW} for the rationally null-homologous knot case. Here, $\tilde{J}$ is defined as a map 
\[\tilde{J}: \underline{\rm Spin^c}(Y,K) \rightarrow \underline{\rm Spin^c}(Y,K)\]
with $\tilde{J}(\xi)=J(\xi)-PD[\mu]$ \footnote{In \cite{OSl}, $\tilde{J}(\xi)=J(\xi)+PD[\mu]$ instead, due to the different knot orientation convention used.}, where $J$ is the conjugation on $\underline{\rm Spin^c}(Y,K)$. Thus for any $\xi \in \underline{\rm Spin^c}(Y,K)$,
\begin{align}
A_{Y,K}(\tilde{J}{\xi})&=A_{Y,K}(J{\xi}-PD[\mu]) \label{Axi=-AJxi}\\
&=\frac{\langle c_1(J\xi)-2PD[\mu], [S]\rangle + [\mu] \cdot [S]}{2 [\mu] \cdot [S]} \notag \\
&=\frac{-\langle c_1(\xi), [S] \rangle -[\mu] \cdot [S]}{2 [\mu] \cdot [S]} \notag\\
&=-A_{Y,K}(\xi). \notag
\end{align}
More concretely, we may understand the above conjugation symmetry using two doubly pointed Heegaard diagrams 
\[\mathcal{H}_1=(\Sigma, \alpha, \beta, w, z) \quad {\rm and} \quad \mathcal{H}_2=(-\Sigma, \beta, \alpha, z, w)\]
for $(Y,K)$. Clearly, the two diagrams $\mathcal{H}_1$ and $\mathcal{H}_2$ give the same generators for $CFK^{\infty}$. 
If $\phi$ is a holomorphic disk in $\mathcal{H}_1$ connecting two generators $\boldsymbol{x}$ to $\boldsymbol{y}$, then there is a holomorphic disk $\bar{\phi}$ connecting $\boldsymbol{x}$ to $\boldsymbol{y}$ in $\mathcal{H}_2$ with the same underlying topological disk, and vice versa. 
Let 
\[\underline{\mathfrak{s}}^1_{w,z}: \mathbb{T}_{\alpha} \cap \mathbb{T}_{\beta} \rightarrow \underline{\rm Spin^c}(Y,K) \quad {\rm and} \quad \underline{\mathfrak{s}}^2_{z,w}: \mathbb{T}_{\alpha} \cap \mathbb{T}_{\beta} \rightarrow \underline{\rm Spin^c}(Y,K)\] 
denote the induced maps to relative $\rm Spin^c$ structures for $\mathcal{H}_1$ and $\mathcal{H}_2$ respectively. Then for any $\boldsymbol{y} \in \mathbb{T}_{\alpha} \cap \mathbb{T}_{\beta}$, 
\[\tilde{J}\underline{\mathfrak{s}}^1_{w,z}(\boldsymbol{y})=\underline{\mathfrak{s}}^2_{z,w}(\boldsymbol{y}) \]
which explains the $\tilde{J}$ map on relative $\rm Spin^c$ structures in the symmetry. As we have also swapped the two basepoints $w$ and $z$ in $\mathcal{H}_1$ and $\mathcal{H}_2$ that leads to the interchanging role of the two filtrations $i$ and $j$, the map $v^-_{\xi}$ is chain homotopy equivalent to $h^-_{\tilde{J}(\xi)}$ for any $\xi \in \underline{\rm Spin^c}(Y,K)$. 
Thus we have 
\begin{lemma}
For any $\xi \in \underline{\rm Spin^c}(Y,K)$, 
\begin{equation}
V^-_{\xi}(K)=H^-_{\tilde{J}(\xi)}(K).
\end{equation}
\end{lemma}
In Rasumussen's notation, this implies
\begin{corollary}
\label{xitjxit}
If $\mathfrak{A}^-_{\xi}=\circ, \ast, + \,\, {\rm or} \,\, -$, then the corresponding $\mathfrak{A}^-_{\tilde{J}\xi}=\circ, \ast, - \,\, {\rm or} \,\, +$ respectively.
\end{corollary}

Finally, \cite[Lemma 3.2]{NW} implies that for any $\xi\in \underline{\rm Spin^c}(Y,K)$, 
$$G_{Y,K}(\tilde{J}\xi)=JG_{Y,K}(\xi)-PD[K]. \footnote{In \cite{NW}, $G_{Y,K}(\tilde{J}\xi)=JG_{Y,K}(\xi)+PD[K]$ instead, due to the different knot orientation convention used.}  $$
Hence, the above conjugation symmetry descends to an isomorphism 
\begin{equation}\label{CFisomorphism}
CF^{\infty}(Y, \mathfrak{s})  \rightarrow CF^{\infty}(Y, J\mathfrak{s} -PD[K]).
\end{equation}

\subsection{The rational slice genus bound for knots whose rational longitude is a framing.}
\begin{theorem}
\label{fsgbound}
Let $K$ be a knot in a rational homology 3-sphere $Y$ whose rational longitude is a framing. Then 
\begin{equation}
\label{nus0g}
\nu^+(Y,K) \leq \|K\|_{Y \times I}^{\partial } + \frac{1}{2},
\end{equation}
and 
\begin{equation}
\label{ds0g}
\mathop{\max}\limits_{\mathfrak{s} \in {\rm Spin^c}(Y)}  \left\{ \frac{1}{2}d(Y,\mathfrak{s})- \frac{1}{2} d(Y, \mathfrak{s}+PD[K]) \right\} \leq \|K\|_{Y \times I}^{\partial } + \frac{1}{2}.
\end{equation}
\end{theorem}

\begin{proof}
Suppose that the knot $K$ has order $p$. Then by assumption, $[\partial S]=[\partial F]=p \lambda_r$. By Corollary \ref{nud}, we have that (\ref{nus0g}) implies (\ref{ds0g}) directly, so it suffices to prove (\ref{nus0g}). Recall that \[\nu^+_\mathfrak{s}(Y,K)={\min}\{A_{Y,K}(\xi)\,|\, \xi \in \underline{\rm Spin^c}(Y,K,\mathfrak{s}) \,\, {\rm and} \,\, V^{\pm}_{\xi}(K)=0\},\]
and 
$$\nu^+(Y,K)=\mathop{\max}\limits_{\mathfrak{s} \in {\rm Spin^c}(Y)} \nu^+_{\mathfrak{s}}(Y,K).$$

Let $\xi_{\rm top}^\mathfrak{s}$ be the relative $\rm Spin^c$ structure that supports $\nu^+(Y,K,\mathfrak{s})$ for each $\mathfrak{s}$, among which $\xi_{\rm top}$ is assumed to be the one that supports $\nu^+(Y,K)$. (If there are more than one such relative $\rm Spin^c$ structures $\xi_{\rm top}^\mathfrak{s}$ with the same Alexander grading when we maximize over $\mathfrak{s} \in {\rm Spin^c}(Y)$, then just pick one of them.)  By definition, $V^-_{\xi_{\rm top}}=0$, so $\mathfrak{A}^-_{\xi_{\rm top}}(Y,K)=+$ or $\circ$ in Rasmussen's notation depending on whether $H^-_{\xi_{\rm top}}>0$ or $H^-_{\xi_{\rm top}}=0$. We split it into two cases.

\medskip
\noindent
\textbf{Case 1}: Suppose $\mathfrak{A}^-_{\xi_{\rm top}}(Y,K)=+$. Clearly, $V^-_{\xi_{{\rm top}-PD[\mu]}}>0$ by the choice of $\xi_{\rm top}$. Lemma \ref{lemmaV-H} implies that $H^-_{\xi_{{\rm top}-PD[\mu]}}>0$, and hence $\mathfrak{A}^-_{\xi_{\rm top}-PD[\mu]}(Y,K)=\ast$. Let 
\[\mathfrak{s}'=G_{Y_{\lambda_r},K_{\lambda_r}}(\xi_{\rm top}-PD[\mu]+i \cdot PD[\lambda_r]) \in {\rm Spin^c}(Y_{\lambda_r})\]
be the underlying $\rm Spin^c$ structure, which is independent of the value of $i \in \mathbb{Z}_p$.
Thus $\mathfrak{A}^-_{\mathfrak{s}'}(Y,K)$ contains a $\ast$, and Lemma \ref{notsurpat} implies that $\mathfrak{D}^-_{\mathfrak{s}'}$ is not surjective. Then we apply Proposition \ref{mappingconeAdjunction} and obtain
\[\frac{1}{2p} \langle c_1(\mathfrak{s}'), [\widehat{S}] \rangle=A_{Y,K}(\xi_{\rm top}-PD[\mu]+i \cdot PD[\lambda_r]) < \frac{g(\widehat{F})}{p}=\frac{g(F)}{p}.\]
Note that
\begin{align*}
A_{Y,K}(\xi_{\rm top}-PD[\mu]+i \cdot PD[\lambda_r])&=A_{Y,K}(\xi_{\rm top}-PD[\mu])\\
&=\frac{\langle c_1(\xi_{\rm top})-2PD[\mu], [S]\rangle + [\mu] \cdot [S]}{2 [\mu] \cdot [S]}\\
&=\frac{\langle c_1(\xi_{\rm top}), [S] \rangle -[\mu] \cdot [S]}{2 [\mu] \cdot [S]}\\
&=A_{Y,K}(\xi_{\rm top})-1. 
\end{align*}
 Thus,
\begin{equation} \label{weakinequality}
\nu^+(Y,K)-1=A_{Y,K}(\xi_{\rm top})-1< \frac{g(F)}{p}.
\end{equation}
We want to tighten the above inequality.  Recall that $\nu^+(Y,K)=A_{Y,K}(\xi_{\rm top})=\dfrac{\langle c_1(\xi_{\rm top}), [S] \rangle + p}{2p}$.  From Turaev \cite[Chapter VI.1.2]{Tur}, we know that $\langle c_1(\xi_{\rm top}), [S] \rangle$ has the same parity as $p$, the number of boundary components of $S$. Hence $p\cdot \nu^+(Y,K)=\dfrac{\langle c_1(\xi_{\rm top}), [S] \rangle + p}{2}$ is an integer. Consequently, (\ref{weakinequality}) can be tighted to
\[\nu^+(Y,K)-\frac{p-1}{p} \leq \frac{g(F)}{p}.\]
Also,
\begin{align*}
\|K\|_{Y \times I}^{\partial }+\frac{1}{2}=\mathop{\min}\limits_{F} \dfrac{-\chi(F)}{2p}+\frac{1}{2}=\mathop{\min}\limits_{F} \frac{-(2-2g(F)-p)+p}{2p}= \mathop{\min}\limits_{F} \frac{g(F)+p-1}{p},
\end{align*}
since $[\partial F]=p \cdot \lambda_r$ has $p$ boundary components. This gives 
\[\nu^+(Y,K) \leq \|K\|_{Y \times I}^{\partial }+\frac{1}{2}.\]

 \bigskip

 \noindent
\textbf{Case 2}: Suppose $\mathfrak{A}^-_{\xi_{\rm top}}(Y,K)=\circ$. Let $\mathfrak{t}=G_{Y,K}(\xi_{\rm top})$, then $\xi_{\rm top}$ is the middle relative $\rm Spin^c$ structure associated to $\mathfrak{t}$, i.e., $\xi_{\rm top}=\xi^0_{\mathfrak{t}}$. We now consider $\mathfrak{A}^-_{\xi^0_{J\mathfrak{t}}}(-Y,-K)$ of the mirror knot, where $\xi^0_{J\mathfrak{t}}$ is the middle relative $\rm Spin^c$ structure associated to the conjugate $\rm Spin^c$ structure $J\mathfrak{t}$. The complex $\mathfrak{A}^-_{\xi^0_{J\mathfrak{t}}}(-Y,-K)=\circ$ or $\ast$, and we prove the two subcases separately. 

\medskip

\medskip
\textbf{Subcase (2.\romannumeral1)}: Suppose $\mathfrak{A}^-_{\xi^0_{J\mathfrak{t}}}(-Y,-K)=\ast$. We will consider the conjugate relative $\rm Spin^c$ structure $\tilde{J}\xi^0_{J\mathfrak{t}}$, and by Corollary \ref{xitjxit}, $\mathfrak{A}^-_{\tilde{J}\xi^0_{J\mathfrak{t}}}(-Y,-K)=\ast$. The same argument as in Case 1 shows that 
\[A_{-Y,-K}(\tilde{J}\xi^0_{J\mathfrak{t}})<\frac{g(F)}{p}<\frac{g(F)+p-1}{p}\]
for any Seifert framed rational slice surface $F$ of $-K$ in $-Y$. Thus
\[A_{-Y,-K}(\tilde{J}\xi^0_{J\mathfrak{t}})<\|-K\|_{-Y \times I}^{\partial }+\frac{1}{2}=\|K\|_{Y \times I}^{\partial }+\frac{1}{2}.\]
By Proposition \ref{middlersc=d}, we have
\begin{align*}
A_{-Y,-K}(\xi^0_{J\mathfrak{t}})&=\frac{1}{2}d(-Y, J\mathfrak{t})-\frac{1}{2}d(-Y, J\mathfrak{t}+PD[-K])\\
&=-\frac{1}{2}d(Y, J\mathfrak{t})+\frac{1}{2}d(Y, J(\mathfrak{t}+PD[K]))\\
&=-\frac{1}{2}d(Y,\mathfrak{t})+\frac{1}{2}d(Y,\mathfrak{t}+PD[K])\\
&=-A_{Y,K}(\xi^0_{\mathfrak{t}}),
\end{align*}
where we used the symmetries $d(-Y,\mathfrak{s})=-d(Y,\mathfrak{s})$ and $d(Y,\mathfrak{s})=d(Y,J\mathfrak{s})$. Thus \[A_{-Y,-K}(\tilde{J}\xi^0_{J\mathfrak{t}})=-A_{-Y,-K}(\xi^0_{J\mathfrak{t}})=A_{Y,K}(\xi^0_{\mathfrak{t}}).\]
This implies
\[\nu^+(Y,K)=A_{Y,K}(\xi^0_{\mathfrak{t}})<\|K\|^{\partial}_{Y \times I}+\frac{1}{2}.\]

\medskip

\textbf{Subcase (2.\romannumeral2)}: Suppose $\mathfrak{A}^-_{\xi^0_{J\mathfrak{t}}}(-Y,-K)=\circ$. In this case, $V^-_{\xi^0_{\mathfrak{t}}}(Y,K)=V^-_{\xi^0_{J\mathfrak{t}}}(-Y,-K)=0$, so we can apply Proposition \ref{V=V=0} and conclude that there is a choice of basis such that $CFK^{\infty}(Y,K,\mathfrak{t})$ has a single generator $x$ supported at $\xi^0_{\mathfrak{t}}$ with no differential in or out. 
In this case, we have that $x \in \widehat{HF}(Y,\mathfrak{t})$ is a nontrival class and 
\[\nu^+(Y,K)=A_{Y,K}(x)=\tau_{x}(Y,K) :=\min \left\{A_{Y,K}(\boldsymbol{x})\,\vert\, \boldsymbol{x} \in \widehat{CF}(Y) \,\, {\rm and} \,\, [\boldsymbol{x}]=x \in \widehat{HF}(Y) \right\}.\]
Our bound (\ref{nus0g}) follows immediately from Hedden-Raoux \cite[Corollary 5.4]{HR}
$$2\tau_x(Y,K) \leq 2 \|K\|_{Y \times I}^\partial+1,$$
and this concludes the proof for Case 2.

\end{proof}

\section{The proof of main theorems}
\label{The proof of main theorems}


In this section, we prove our main theorems. Our strategy is to apply Ozsv\'{a}th-Szab\'{o}'s trick of Morse surgery and convert all knots whose rational longitude is not a framing into the framing case.

\subsection{Ozsv\'{a}th-Szab\'{o}'s trick of Morse surgery} 
\label{trick of Morse surgery}

For any knot $K$ of order $p$ in a rational homology 3-sphere $Y$, its rational longitude can be written as 
\[\lambda_r=p' \lambda + q' \mu \,\, {\rm for \,\, some} \,\, p', q' \in \mathbb{Z},\] 
where $\lambda$ is a framing we fix. Note that $p$ is a multiple of $p'$, and
\[[\partial F]=[\partial S]=k\lambda_r \,\, {\rm with} \,\, k=p/p'.\]

Let $q' \slash p'=m- n \slash p'$ with $0 \leq n < p'$ and $n, m \in \mathbb{Z}$. Applying the Slam-Dunk move, $Y_{q' \slash p'}(K)$ can be realized by surgery with coefficient $m$ along the knot 
\[K'=K \# O_{p' \slash n} \subset Y'=Y \# L(p',n),\] 
where $O_{p' \slash n} \subset L(p',n)$ is the so-called \textit{$U$-knot} in the lens space $L(p',n)$, which is obtained by viewing one component of the Hopf link as a knot inside $L(p',n)$, thought of as $p' \slash n$-surgery on the other component of the Hopf link. See Figure \ref{Uknot}. Let $\lambda'_r$ be the rational longitude for $K'$, then $\lambda'_r$ has slope $m$, so it is a framing. Note that $K'$ also has order $p$.  This trick of Morse surgery enables us to reduce all cases to the case we discussed in the previous section.

We claim that
\begin{equation}
\label{gFleqgF'}
\mathop{\min}\limits_{F'} g(F') \leq \mathop{\min}\limits_{F} g(F),
\end{equation}
where $F$ and $F'$ denote Seifert framed rational slice surfaces for $K$ and $K'$ respectively.  Note that the meridian disk in the glued solid torus along the other Hope link component on which we perform the $p'/n$-surgery is a rational Seifert surface of $O_{p' \slash n}$, and we denote this disk by $D_{p' \slash n}$. 
Then for any Seifert framed rational slice surface $F$ for $K$, we can construct a Seifert framed rational slice surface $F'$ for $K'$ with $g(F')=g(F)$ by performing a band sum of $k$ copies of $D_{p' \slash n}$ and one copy of $F$ along $p$ bands.  This proves (\ref{gFleqgF'}).

\medskip
The next step is to relate the knot Floer homology of $K$ and $K'$.  The lens space $L(p',n)$ can be represented by a standard Heegaard diagram of genus one. There are $p'$ intersection points $x_{j}$ of the $\boldsymbol{\alpha}$ and $\boldsymbol{\beta}$ curves in the Heegaard disgram, and they represent $p'$ different $\rm Spin^c$ structures, denoted by $\mathfrak{t}_{j}$ with $j\in \mathbb{Z}_{p'}$. The $U$-knot $O_{p' \slash n}$ is an example of Floer simple knots whose knot Floer complex $CFK^{\infty}(L(p',n), O_{p' \slash n}, \mathfrak{t}_{j})$ is generated by a single generator $x_j$. 
See Figure \ref{Heegaard diagram} for the (doubly) pointed Heegaard diagram for $L(5,1)$ and $O_{5/1}$.  Thus, K\"{u}nneth formula implies:


\begin{figure}[t]
\centering
\subfigure[The $U$-knot $O_{p'/n} \subset L(p',n)$.]{\label{Uknot}
\includegraphics[width=0.48\textwidth]{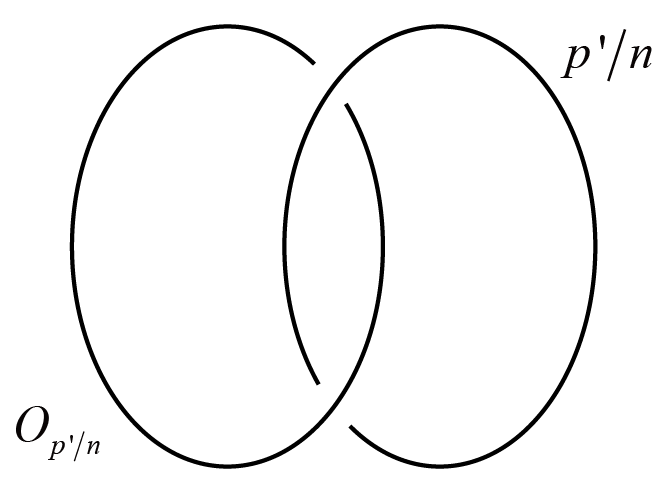}}\quad\quad
\subfigure[The (doubly) pointed Heegaard diagram for $L(5,1)$ and $O_{5/1}$.]{\label{Heegaard diagram}
\includegraphics[width=0.42\textwidth]{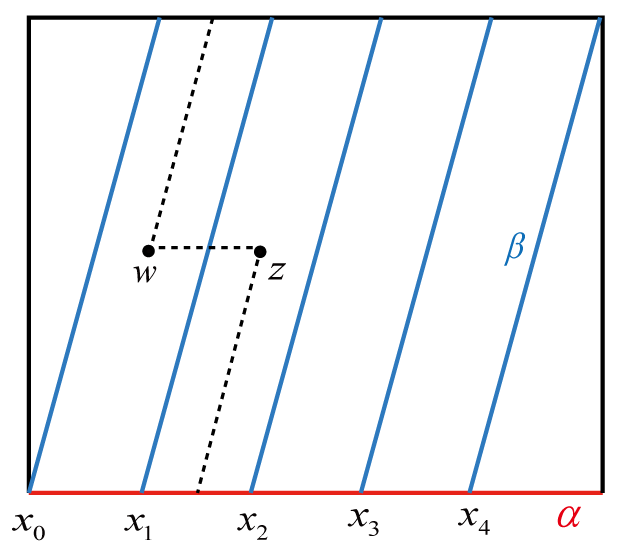}}
\caption{}
\label{UknotH}
\end{figure}

\begin{lemma}{\rm \cite[Corollary 5.3]{OSr} \cite[Lemma 3.8]{Rao}}
For any $\rm Spin^c$ structure $\mathfrak{s} \in {\rm Spin^c}(Y)$ and any $\mathfrak{t}_{j} \in {\rm Spin^c}(L(p',n))$ with $j \in \mathbb{Z}_{p'}$,
\[CFK^{\infty}(Y',K', \mathfrak{s} \# \mathfrak{t}_{j}) \cong CFK^{\infty}(Y,K,\mathfrak{s})\]
with the Alexander grading shifted by $A_{L(p',n), O_{p' \slash n}}(x_{j})$.  
\label{CFKiso}
\end{lemma}

\begin{lemma}
\label{AUknot}
Denote $A_{\max}$ and $A_{\min}$ the maximum and minimum of the $p'$ Alexander gradings $A_{L(p',n), O_{p' \slash n}}(x_{j})$ with $j \in \mathbb{Z}_{p'}$, respectively. Then 
\begin{equation}
\label{Amax}
A_{\max}=-A_{\min}=\frac{p'-1}{2p'}.
\end{equation}

\end{lemma}

\begin{proof}
As discussed earlier, the disk $D_{p' \slash n}$ is a rational Seifert surface of the $U$-knot $O_{p' \slash n}$. By Theorem \ref{rknotgenus}, 
\[A_{\max}-A_{\min}=\dfrac{-\chi(D_{p' \slash n})+\left|[\partial D_{p' \slash n}] \cdot [\mu] \right|}{\left|[\partial D_{p' \slash n}] \cdot [\mu] \right|}=\dfrac{-1+p'}{p'}.\]
Equation (\ref{Amax}) then follows from the symmetry of the Alexander grading.  
\end{proof}

\begin{remark}
With a little extra work, it can be shown that the set $\{A_{L(p',n), O_{p' \slash n}}(x_{j})\}_{j\in \mathbb{Z}_{p'}}$ is in fact an arithmetic progression between $A_{\min}=-\frac{p'-1}{2p'}$ and $A_{\max}=\frac{p'-1}{2p'}$.

\end{remark}

\begin{corollary}
\label{nu+difference}

\begin{equation}
\label{nuK'nuK}
\nu^+(Y',K')=\nu^+(Y,K)+\frac{p'-1}{2p'}.
\end{equation}
\end{corollary}

\begin{proof}
From Lemma \ref{CFKiso}, we see $\nu^+(Y', K', \mathfrak{s} \# \mathfrak{t}_{j})=\nu^+(Y, K, \mathfrak{s})+A_{L(p',n), O_{p' \slash n}}(x_{j})$.  Maximizing over all $\rm Spin^c$ structures and then applying Lemma \ref{AUknot} gives (\ref{nuK'nuK}).
\end{proof}

\subsection{The proof of Theorems \ref{d-dleqprsg} and \ref{nuleqprsg}}

\begin{proof}[Proof of Theorem \ref{nuleqprsg}]
Note that $K'$ is a knot of order $p$ whose rational longitude is a framing. Thus, we can apply Theorem \ref{fsgbound} and get
\begin{equation}
\label{nuY'K'p}
\nu^+(Y',K') \leq \|K'\|^{\partial}_{Y' \times I} + \frac{1}{2}=\mathop{\min}\limits_{F'} \frac{g(F')+p-1}{p}.
\end{equation}

Now, we consider the rational slice genus of $K$.  Since a Seifert framed rational slice surface $F$ of $K$ has $k=p/p'$ boundary components, 
\[\|K\|_{Y \times I}^{\partial}+\frac{1}{2}=\mathop{\min}\limits_{F} \dfrac{-\chi(F)}{2p}+\frac{1}{2}=\mathop{\min}\limits_{F} \frac{-(2-2g(F)-k)}{2p}+\frac{1}{2}=\mathop{\min}\limits_{F} \frac{g(F)}{p} + \frac{p'+1}{2p'}-\frac{1}{p}.\]
By (\ref{gFleqgF'}),
$$\|K'\|_{Y' \times I}^{\partial }+\frac{1}{2}\leq \mathop{\min}\limits_{F} \frac{g(F)+p-1}{p}=\|K\|_{Y \times I}^{\partial }+\frac{1}{2}+\frac{p'-1}{2p'}$$
This combined with (\ref{nuK'nuK}) and (\ref{nuY'K'p}) gives

$$\nu^+(Y, K) =\nu^+(Y',K')-\frac{p'-1}{2p'}\leq |K'\|_{Y' \times I}^{\partial}+\frac{1}{2}-\frac{p'-1}{2p'} 
\leq \|K\|_{Y \times I}^{\partial} + \frac{1}{2}.$$

\medskip
\noindent
Finally, for a Floer simple knot $K$, Theorem \ref{Rasmussen conjecture NW} implies the inequalities of opposite direction:
$$\nu^+(Y,K) \geq \mathop{\max}\limits_{\mathfrak{s} \in {\rm Spin^c}(Y)} \left\{ \frac{1}{2}d(Y,\mathfrak{s})-\frac{1}{2}d(Y,\mathfrak{s}+PD[K]) \right\}=\|K\|_Y + \frac{1}{2} \geq \|K\|_{Y \times I}^{\partial}+\frac{1}{2}$$
Thus, it must attain the equality everywhere.
\end{proof}

Theorem \ref{d-dleqprsg} is an immediate corollary of Theorem \ref{nuleqprsg} and Corollary \ref{nud}.


\begin{thebibliography}{10}

\bibitem{Baker} K. Baker, \textit{Small genus knots in lens spaces have small bridge number}, Alg. Geom. Topol., \textbf{6}: 1519-1621, 2006.
\bibitem{Berge} J. Berge, \textit{Some knots with surgeries yielding lens spaces}, unpublished manuscript.
\bibitem{CG} D. Calegari and C. Gordon, \textit{Knots with small rational genus.} Comment. Math. Helv. 88(1): 85-130, 2013.
\bibitem{GN} J. Greene and Y. Ni, \textit{Non-simple genus minimizers in lens spaces}, to appear Alg. Geom. Topol. 
\bibitem{Hedden} M. Hedden, \textit{On Floer homology and the Berge conjecture on knots admitting lens space surgeries.} Trans. Amer. Math. Soc. \textbf{363}(2): 949-968, 2011 
\bibitem{HL} M. Hedden and A. Levine, \textit{A surgery formula for knot Floer homology.} arXiv:1901.02488.
\bibitem{HR2} M. Hedden and K. Raoux, \textit{Knot Floer homology and relative adjunction inequalities.} arXiv:2009.05462.
\bibitem{HR} M. Hedden and K. Raoux, \textit{A 4-dimensional rational genus bound.} arXiv:2308.16853.
\bibitem{Hom} J. Hom, \textit{A survey on Heegaard Floer homology and concordance.} J. Knot Theory Ramif., \textbf{26}(02): 1740015, 2017.
\bibitem{HomL} J. Hom and T. Lidman, \textit{A note on positive-definite, symplectic four-manifolds.} J. Eur. Math. Soc., \textbf{21}(1): 257-270, 2018.
\bibitem{HW} J. Hom and Z. Wu, \textit{Four-ball genus bounds and a refinement of the Ozsv\'ath-Szab\'o tau-invarian.} J. Symp. Geom. 14 (1): 305-323, 2016.
\bibitem{LRS} A. Levine, D. Ruberman and S. Strle, \textit{Non-orientable surfaces in homology cobordisms.} Geom. Topol., \textbf{19}(1): 439-494, 2015.
\bibitem{MO} C. Manolescu and P. Ozsv\'{a}th, \textit{Heegaard Floer homology and integer surgeries on links.} arXiv:1011.1317v5.
\bibitem{Ni} Y. Ni, \textit{Link Floer homology detects the Thurston norm.} Geom. Topol., \textbf{13}(5): 2991-3019, 2009.
\bibitem{NV} Y. Ni and F. Vafaee, \textit{Null surgery on knots in L-spaces.} Trans. Am. Math. Soc., \textbf{372}(12): 8279-8306, 2019.
\bibitem{NW} Y. Ni and Z. Wu, \textit{Heegaard Floer correction terms and rational genus bounds.} Adv. Math., \textbf{267}: 360-380, 2014.
\bibitem{OSpa} P. Ozsv\'{a}th and Z. Szab\'{o}, \textit{Holomorphic disks and three-manifold invariants: properties and applications.} Ann. Math., \textbf{159}(3): 1159-1245, 2004.
\bibitem{OSk} P. Ozsv\'{a}th and Z. Szab\'{o}, \textit{Holomorphic disks and knot invariants.} Adv. Math., \textbf{186}(1): 58-116, 2004.
\bibitem{OSc} P. Ozsv\'{a}th and Z. Szab\'{o}, \textit{Holomorphic triangles and invariants for smooth four-manifolds.} Adv. Math., \textbf{202}(2): 326-400, 2006.
\bibitem{OSl} P. Ozsv\'{a}th and Z. Szab\'{o}, \textit{Holomorphic disks, link invariants and the multi-variable Alexander polynomial.} Alg. Geom. Topol., \textbf{8}(2): 615-692, 2008.
\bibitem{OSr} P. Ozsv\'{a}th and Z. Szab\'{o},  \textit{Knot Floer homology and rational surgeries.} Alg. Geom. Topol., \textbf{11}(1): 1-68, 2010.
\bibitem{Rao} K. Raoux, \textit{$\tau$–invariants for knots in rational homology spheres.} Alg. Geom. Topol., \textbf{20}(4): 1601-1640, 2020.
\bibitem{Ras} J. Rasmussen, \textit{Floer homology and knot complements.} PhD thesis, Harvard University, 2003.
\bibitem{RasNotation} J. Rasmussen, \textit{Lens space surgeries and L-space homology spheres.} arXiv:0710.2531
\bibitem{Tur} V. Turaev, \textit{Torsions of 3-dimensional manifolds.} Progress in Math. 208, Birkh\"{a}user Verlag, Basel, 2002.
\bibitem{Tur2} V. Turaev, \textit{A function on the homology of 3–manifolds.} Alg. Geom. Topol., \textbf{7}(1): 135-156, 2007.
\bibitem{Yang} J. Yang, \textit{Distance one surgeries on the lens space $L(n,1)$.} arXiv:2108.06199, 2021.

\end{thebibliography}
\end{document}